\documentclass[12pt]{amsart}
\usepackage{bbm}
\usepackage[pdftex]{graphicx}
\usepackage{amscd, amssymb, dsfont, upgreek, mathrsfs}
\input{xy}
\xyoption{all}

\newtheorem{Thm}{Theorem}

\newtheorem{Prop}[Thm]{Proposition}
\newtheorem{Def}[Thm]{Definition}
\newtheorem{Def/Thm}[Thm]{Definition/Theorem}
\newtheorem{Cor}[Thm]{Corollary}
\newtheorem{Lemma}[Thm]{Lemma}

\theoremstyle{remark}

\newcommand{\F}{{\mathsf{M}}}
\newcommand{\ti }{\times}
\newcommand{\ot }{\otimes}
\newcommand{\ra }{\rightarrow}

\newcommand{\lra }{\longrightarrow}

\newcommand{\Spec}{{\mathrm{Spec}}}

\newcommand{\lann}{\langle\langle}
\newcommand{\rann}{\rangle\rangle}
\newcommand{\lannn}{\left\langle\left\langle}
\newcommand{\rannn}{\right\rangle\right\rangle}

\newcommand{\G}{{\bf G}}

\newcommand{\PP }{{\mathbb P}}

\newcommand{\QQ }{{\mathbb Q}}
\newcommand{\CC }{{\mathbb C}}
\newcommand{\ZZ }{{\mathbb Z}}

\newcommand{\vir}{\mathrm{vir}}

\newcommand{\DD}{\mathsf{D}}

\newcommand{\T}{{\mathsf{T}}}

\newcommand{\lan}{\langle}
\newcommand{\ran}{\rangle}

\newcommand{\dsI}{\mathds{I}}

\newcommand{\pP}{\mathsf{P}}

\newcommand{\ppl}{{\mathsf{P}}\left[}
\newcommand{\ppr}{\right]}

\newcommand{\rarr}{\longrightarrow}

\def \proj {{\mathbb{P}}}
\newcommand{\com}{{\mathbb{C}}}

\newcommand{\oh}{{\mathcal O}}

\begin{document}

\title[Stable quotients and the holomorphic anomaly equation]
{Stable quotients and the holomorphic anomaly equation}

\author{Hyenho Lho}
\address{Department of Mathematics, ETH Z\"urich}
\email {hyenho.lho@math.ethz.ch}
\author{Rahul Pandharipande}
\address{Department of Mathematics, ETH Z\"urich}
\email {rahul@math.ethz.ch}
\date{March 2018}.

\begin{abstract} 
We study the fundamental relationship between stable quotient
invariants and the B-model for local $\PP^2$
in all genera. Our main result is a direct geometric
proof of the holomorphic anomaly equation in the precise form predicted by B-model physics. The method
yields new holomorphic anomaly equations
for an infinite class of twisted
theories on projective spaces.

An example of such a
twisted theory is
the formal quintic  defined by a
hyperplane section 
of $\mathbb{P}^4$ in all genera via the Euler class of
a complex.
The formal quintic theory is found to satisfy the holomorphic anomaly equations conjectured for the true
quintic theory. Therefore,
the formal quintic theory  and the true quintic theory
should be related by transformations which respect the
holomorphic anomaly equations.

 \end{abstract}

\maketitle

\setcounter{tocdepth}{1} 
\tableofcontents

\setcounter{section}{-1}

\section{Introduction}

\subsection{GW/SQ} 
To the A-model Gromov-Witten theory of a Calabi-Yau 3-fold $X$ is 
conjecturally associated
the B-model theory of a mirror Calabi-Yau 3-fold ${Y}$.
At the genus 0 level, much is known about the geometry underlying the
B-model: Hodge theory, period integrals, and the linear sigma model.
In higher genus, mathematical techniques have been
less successful. In toric Calabi-Yau geometries,
the topological
recursion
of Eynard and Orantin provides
a link to
 the B-model \cite{BKMP,EMO}.
Another mathematical approach to
the B-model (without the
toric hypothesis) 
has been proposed by Costello and Li \cite{CosLi} 
following paths suggested by string theory. A very different mathematical view of the geometry of B-model
invariants is pursued here.

Let 
$X_5\subset \mathbb{P}^4$
be a (nonsingular) quintic Calabi-Yau 3-fold.
The moduli space
of stable maps to the quintic
of genus $g$ and degree $d$,
$$\overline{M}_g(X_5,d)\subset \overline{M}_g(\mathbb{P}^4,d)\, ,$$
has virtual dimension 0.
The Gromov-Witten invariants{\footnote{In degree 0, the moduli 
spaces of maps and stable quotients are both empty for genus 0 and 1,
so the invariants in these cases vanish.}},
\begin{equation}\label{fredfredgw}
N^{\mathsf{GW}}_{g,d}\, =\,
\langle 1\rangle_{g,d}^{\mathsf{GW}}\,=\, \int_{[\overline{M}_g(X_5,d)]^{vir}} 1\, ,
\end{equation}
have been studied for more than 20 years, 
see \cite{CKatz,FP,Kon} for an
introduction to the subject.

The theory of stable quotients developed in \cite{MOP}
was partially inspired by the question of finding a geometric approach
to a higher genus linear sigma model. 
The moduli space
of stable quotients for the quintic,
$$\overline{Q}_g(X_5,d)\subset \overline{Q}_g(\mathbb{P}^4,d)\, ,$$
was defined in \cite[Section 9]{MOP}, and questions about the associated
integral theory,
\begin{equation}
\label{fredfred}
N^{\mathsf{SQ}}_{g,d}
\,=\, \langle 1 \rangle_{g,d}^{\mathsf{SQ}}
\, =\, \int_{
[\overline{Q}_g(X_5,d)]^{vir}} 1\, ,
\end{equation}
were posed.

The existence of a natural obstruction theory on
$\overline{Q}_g(X_5,d)$
and a virtual
fundamental class ${[\overline{Q}_g(X_5,d)]^{vir}}$ is easily seen{\footnote{For stability,
marked points are required in genus 0 and
positive degree is required in genus 1.}}
in genus 0 and 1.
A proposal in higher genus for the obstruction theory and
virtual class was made in \cite{MOP} and was carried out in
significantly greater generality in the setting of quasimaps in
\cite{CKM}.
In genus 0 and 1, the integral theory \eqref{fredfred}
 was calculated 
in \cite{CZ} and \cite{KL} respectively.
The answers on the stable quotient side {\em exactly}
match the string theoretic B-model for the quintic in
genus 0 and 1.

A relationship in every genus between the
Gromov-Witten and stable quotient invariants of
the quintic has been recently proven by Ciocan-Fontanine
and Kim \cite{CKw}.{\footnote{A second proof (in most cases) can
be found in  \cite{CJR}.}}
Let $H\in H^2(X_5,\mathbb{Z})$ be the hyperplane class of
the quintic,
and let
$$\mathcal{F}_{g,n}^{\mathsf{GW}}(Q)\, =\, \langle \, \underbrace{H,\ldots,H}_{n}  \, \rangle_{g,n}^{\mathsf{GW}} \, =\, 
\sum_{d=0}^\infty  Q^d
\int_{[\overline{M}_{g,n}(X_5,d)]^{vir}} \prod_{i=1}^n
\text{ev}_i^*(H)\,  ,$$
$$\mathcal{F}_{g,n}^{\mathsf{SQ}}(q) \, =\, \langle \, \underbrace{H,\ldots,H}_{n}  \, \rangle_{g,n}^{\mathsf{SQ}} \, =\, 
\sum_{d=0}^\infty  q^d
\int_{[\overline{Q}_{g,n}(X_5,d)]^{vir}} \prod_{i=1}^n
\text{ev}_i^*(H)\, $$
be the Gromov-Witten and stable quotient series
respectively (involving the pointed moduli spaces and the
evaluation morphisms at the markings).
Let 
$$I^{\mathsf{Q}}_0(q)=\sum_{d=0}^\infty q^d \frac{(5d)!}{(d!)^5}\, ,  \ \ \ 
I^{\mathsf{Q}}_1(q)= \log(q)I^{\mathsf{Q}}_0(q) + 5 \sum_{d=1}^\infty q^d \frac{(5d)!}{(d!)^5} 
\left( \sum_{r=d+1}^{5d} \frac{1}{r}\right)\, .
$$
The mirror map is defined by
$$Q(q) = \exp\left(\frac{I^{\mathsf{Q}}_1(q)}{I^{\mathsf{Q}}_0(q)}\right) =
q\cdot \exp\left( \frac{5 \sum_{d=1}^\infty q^d \frac{(5d)!}{(d!)^5} 
\left( \sum_{r=d+1}^{5d} \frac{1}{r}\right)}{\sum_{d=0}^\infty q^d \frac{(5d)!}{(d!)^5}} \right)
\, .$$
The relationship  between the Gromov-Witten and stable quotient
invariants of the quintic in case $$2g-2+n>0$$ is given by the following result
\cite{CKw}:
\begin{equation}\label{345}
\mathcal{F}_{g,n}^{\mathsf{GW}}(Q(q)) =
I^{\mathsf{Q}}_0(q)^{2g-2+n} \cdot \mathcal{F}_{g,n}^{\mathsf{SQ}}(q)\, .
\end{equation}
The transformation \eqref{345} shows the stable quotient
theory matches the string theoretic
B-model series for the quintic $X_5$.

Based on the above investigation of the quintic, 
a fundamental relationship between
the stable quotient invariants and the B-model for all
local and complete intersection Calabi-Yau 3-folds $X$
is natural to propose: 
{\it the stable quotient invariants of 
$X$ exactly  
equal the B-model invariants of the mirror $Y$}.

\subsection{Holomorphic anomaly}\label{ham}
The (conjectural) holomorphic anomaly equation is a beautiful property
of the string theoretic B-model series which has been used
effectively since \cite{BCOV}. Since the stable quotients
invariants provide a geometric proposal for the B-model series,
we should look for the geometry of the holomorphic
anomaly equation in the moduli space of stable quotients.

Since the stable quotients invariant for the quintic are
difficult{\footnote{There
are very few mathematical derivations of the
holomorphic anomaly equation in
higher genus for compact Calabi-Yau 3-folds. A genus 2 result
for the Enriques Calabi-Yau
can be found in \cite{MPE}.} to approach in $g\geq 2$, 
we study here instead the stable quotient invariants in all genera
for local $\PP^2$ -- the toric
Calabi-Yau 3-fold given by the
total space $K\PP^2$ of
the canonical bundle over
$\PP^2$ . Our main result is a direct proof of the
holomorphic anomaly equations in the precise form predicted by 
B-model physics for local $\PP^2$.

The holomorphic anomaly
equation can be proven
for toric Calabi-Yau
3-folds by the 
topological
recursion
of Eynard and Orantin \cite{EMO}
together with the
remodelling conjecture
 \cite{BKMP} proven 
 recently in \cite{FLZ}.
 The path of \cite{BKMP,EMO,FLZ} to
 the holomorphic anomaly equation
 for toric Calabi-Yau 3-folds
 involves explicit manipulation of 
 the Gromov-Witten
 partition function
 (which can be computed
 by now in several different
 ways in toric Calabi-Yau cases \cite{TopVer,gwdt}). 
 Our approach involves
 the geometry of the moduli space of 
 stable quotients ---
 the derivation of the
 holomorphic anomaly equation
 takes place on the stable quotient side. 
 
The new perspective 
yields new holomorphic anomaly equations
for an infinite class of twisted
theories on projective spaces. Some
of these are
related to toric Calabi-Yau geometries
(such as local $\PP^1 \times \PP^1$),
but most are not. The holomorphic
anomaly equations for these
theories have never been considered
before.

 A particular twisted theory
 on $\PP^4$ is
 related to the quintic
 3-fold. Let the algebraic torus
$$\mathsf{T}_5=(\com^*)^5$$
 act with the standard linearization
 on $\PP^4$ with weights
 $\lambda_0,\ldots,\lambda_4$
 on the vector space  $H^0(\PP^4,\mathcal{O}_{\PP^4}(1))$.
Let
\begin{equation}\label{gred}
\mathsf{C} \rightarrow \overline{M}_g(\mathbb{P}^4,d)
\, , \ \ \ 
f:\mathsf{C}\rightarrow \PP^4
\, , \ \ \ 
\mathsf{S}=f^*\mathcal{O}_{\PP^4}(-1) \rightarrow \mathsf{C}\, 
\end{equation}
be the universal curve,
the universal map,
and the universal bundle over
the moduli space of stable maps
--- all
equipped with canonical $\T_5$-actions.
We define the {\em formal quintic invariants} by{\footnote{The negative
exponent denotes the dual:
$\mathsf{S}$ is a line bundle and $\mathsf{S}^{-5}=(\mathsf{S}^\star)^{\otimes 5}$.}}
\begin{equation}\label{fredfredfred}
\widetilde{N}_{g,d}^{\mathsf{GW}}= \int_{[\overline{M}_g(\mathbb{P}^4,d)]^{vir}} e(R\pi_*(\mathsf{S}^{-5}))\, , \   
\end{equation}
where $e(R\pi_*(\mathsf{S}^{-5}))$
is the equivariant Euler class (defined
after localization).
The integral \eqref{fredfredfred} is
homogeneous of degree 0 in localized
equivariant cohomology,
$$\int_{[\overline{M}_g(\mathbb{P}^4,d)]^{vir}} e(R\pi_*(\mathsf{S}^{- 5}))
\, \in \,{\mathbb{Q}}(\lambda_0,\ldots,\lambda_4)\, ,$$
and defines a rational number 
 $\widetilde{N}_{g,d}^{\mathsf{GW}}\in {\mathbb{Q}}$
after the
specialization
$$\lambda_i = \zeta^i \lambda_0$$
for a primitive fifth root of unity
$\zeta^5=1$.

 
A main result in the sequel \cite{LP2} proves  that the   holomorphic anomaly equations
conjectured for the
true quintic theory \eqref{fredfredgw}
are satisfied by the formal 
quintic theory
\eqref{fredfredfred}.
%
In particular,
the formal quintic theory  and the true quintic theory
should be related by transformations which respect the
holomorphic anomaly equations.

\subsection{Twisted theories on
$\PP^m$}\label{twth}
Twisted theories
associated to $\proj^m$
can be constructed as follows. 
Let the algebraic torus
$$\mathsf{T}_{m+1}=(\com^*)^{m+1}$$
 act with the standard linearization 
 on $\PP^m$ with weights
 $\lambda_0,\ldots,\lambda_m$
 on the vector space  $H^0(\PP^m,\mathcal{O}_{\PP^m}(1))$.

Let 
$
\overline{M}_{g}(\proj^m,d)
$ 
be the moduli space of 
stable maps to $\proj^m$ equipped with the canonical
$\T_{m+1}$-action, and 
let
$$
\mathsf{C} \rightarrow \overline{M}_g(\mathbb{P}^m,d)
\, , \ \ \ 
f:\mathsf{C}\rightarrow \PP^m
\, , \ \ \ 
\mathsf{S}=f^*\mathcal{O}_{\PP^m}(-1) \rightarrow \mathsf{C}\,
$$
be the standard universal structures. 
Let 
$$\mathsf{a}=(a_1,\ldots, a_r)\, ,\ \ \ \ \mathsf{b}=(b_1,\ldots,b_s)$$
be vectors of  {\em positive} integers satisfying
the conditions
$$\sum_{i=1}^r a_i - \sum_{j=1}^s b_j = m+1 \ \ \ \text{and} \ \ \ m-r+s=3\, .$$
The first is the Calabi-Yau condition
and the second is the dimension 3 condition.

The Gromov-Witten invariants of 
the {\em $(\mathsf{a},\mathsf{b})$-twisted 
geometry of $\PP^m$} 
are defined via the equivariant integrals
\begin{equation}\label{twt}
\widetilde{N}_{g,d}^{\mathsf{GW}} = \int_{[\overline{M}_g(\proj^m,d)]^{vir}}
\prod_{i=1}^r e(R\pi_*\mathsf{S}^{-a_i})
\prod_{j=1}^s e(-R\pi_*\mathsf{S}^{b_j})
\, .
\end{equation}
The integral \eqref{twt} is
homogeneous of degree 0 in localized
equivariant cohomology
and defines a rational number 
 $$\widetilde{N}_{g,d}^{\mathsf{GW}}\in {\mathbb{Q}}$$
after the
specialization
$$\lambda_i = \zeta^i \lambda_0$$
for a primitive $(m+1)$-root of unity
$\zeta^{m+1}=1$.

The standard
theory of local $\PP^2$ theory is recovered
in case 
$$m=2\, , \ \ r=0\,, \ \  \mathsf{b}=(3)\, .$$
We will see in \cite{LP3} that the case
$$m=3\,, \ \ \mathsf{a}=(2)\,, \ \ \mathsf{b}=(2)$$
is related to the local $\PP^1 \times 
\PP^1$ geometry. The formal quintic
theory arises in case
$$m=4\,, \ \ \mathsf{a}=(5)\,, \ \ s=0\,.$$


The stable quotient
perspective of the paper yields holomorphic anomaly equations for all
$(\mathsf{a},\mathsf{b})$-twisted
theories on $\PP^m$
satisfying the Calabi-Yau and
dimension 3 conditions. 
Our goal here
is to present two of the most
interesting cases.



\subsection{Holomorphic anomaly for $K\PP^2$} \label{holp2}

We state here the precise form of the holomorphic
anomaly equations for local $\proj^2$.

Let $H\in H^2(K\PP^2,\mathbb{Z})$ be the hyperplane class of
the obtained from $\PP^2$,
and let
$$\mathcal{F}_{g,n}^{\mathsf{GW}}(Q)\, =\, \langle \, \underbrace{H,\ldots,H}_{n}  \, \rangle_{g,n}^{\mathsf{GW}} \, =\, 
\sum_{d=0}^\infty  Q^d
\int_{[\overline{M}_{g,n}(K\PP^2,d)]^{vir}} \prod_{i=1}^n
\text{ev}_i^*(H)\,  ,$$
$$\mathcal{F}_{g,n}^{\mathsf{SQ}}(q) \, =\, \langle \, \underbrace{H,\ldots,H}_{n}  \, \rangle_{g,n}^{\mathsf{SQ}} \, =\, 
\sum_{d=0}^\infty  q^d
\int_{[\overline{Q}_{g,n}(K\PP^2,d)]^{vir}} \prod_{i=1}^n
\text{ev}_i^*(H)\, $$
be the Gromov-Witten and stable quotient series
respectively (involving the
evaluation morphisms at the markings).
The relationship  between the Gromov-Witten and stable quotient
invariants of $K\PP^2$ in case $2g-2+n>0$ is proven in 
\cite{CKg}:
\begin{equation}\label{3456}
\mathcal{F}_{g,n}^{\mathsf{GW}}(Q(q)) =
 \mathcal{F}_{g,n}^{\mathsf{SQ}}(q),
\end{equation}
where $Q(q)$ is the mirror map,
$$
I^{K\proj^2}_1(q)=
\log(q) + 3 \sum_{d=1}^\infty (-q)^d \frac{(3d-1)!}{(d!)^3} \, ,$$
$$Q(q)=\exp\left(I^{K\PP^2}_1(q)\right) =
q\cdot \exp\left(3 \sum_{d=1}^\infty (-q)^d \frac{(3d-1)!}{(d!)^3}\right) \, .$$
Again, the stable quotient
theory matches the string theoretic
B-model series for $K\PP^2$.

In order to state the holomorphic anomaly equations,
we require the following additional series in $q$.
\begin{eqnarray*}
L(q)&=& (1+27 q)^{-\frac{1}{3}}\, = \, 1-9q+162 q^2 + \ldots\, ,\\
C_1(q)&=& q \frac{d}{dq} I_1^{K\proj^2}\, ,\\
A_2(q)&=& \frac{1}{L^3}\left(3\frac{q\frac{d}{dq}C_1}{C_1}
+1 -\frac{L^3}{2}\right)\, .
\end{eqnarray*}
Let $T$ be the standard coordinate mirror to $t=\log(q)$,
\begin{align}\label{MM}T= I_1^{K\proj^2}(q)\, .
\end{align}
Then $Q(q)=\exp(T)$ is the mirror map.

The ring $\mathbb{C}[L^{\pm1}]=\mathbb{C}[L,L^{-1}]$ will play basic role.
Since 
$$q = \frac{1}{27}\left(L^{-3}-1\right),$$
we have $\mathbb{C}[q] \subset \mathbb{C}[L^{\pm1}]$.
For the holomorphic anomaly equation, consider the
free polynomial rings in the variables $A_2$ and $C_1^{-1}$ over $\com[L^{\pm1}]$,
\begin{equation}
\label{dsdsds7}
\mathbb{C}[L^{\pm1}][A_2]\, , \ \ \ 
\mathbb{C}[L^{\pm1}][A_2,C_1^{-1}]\,.
\end{equation}
There are canonical maps
\begin{equation}\label{dsdsds22}
\mathbb{C}[L^{\pm1}][A_2]\rightarrow \mathbb{C}[[q]]\, , \ \ \ 
\mathbb{C}[L^{\pm1}][A_2,C_1^{-1}]\rightarrow \mathbb{C}[[q]]
\end{equation}
given by assigning the above defined series $A_2(q)$ and
$C_1^{-1}(q)$ to the variables $A_2$ and $C_1^{-1}$ respectively.
We may therefore consider elements of the rings \eqref{dsdsds7} {\em either} 
as free polynomials in the variables $A_2$ and $C_{1}^{-1}$ {\em or}
as series in $q$.

Let $F(q)\in \mathbb{Q}[[q]]$ be a series in $q$. When we write
$$F(q) \in \mathbb{C}[L^{\pm1}][A_2]\, ,$$
we mean there is a canonical lift ${F}\in \mathbb{C}[L^{\pm1}][A_2]$
for which
$${F} \mapsto F(q) \in \mathbb{C}[[q]]$$
under the map \eqref{dsdsds22}. 
The symbol $F$ {\em without the
argument $q$} is the lift.
The notation 
$$F(q) \in \mathbb{C}[L^{\pm1}][A_2,C_1^{-1}]$$
is parallel.


\begin{Thm} \label{ooo} For the stable quotient invariants of $K\PP^2$,  
\begin{enumerate}
\item[(i)]
$\mathcal{F}_g^{\mathsf{SQ}} (q) \in \mathbb{C}[L^{\pm1}][A_2]$
for  $g\geq 2$, \vspace{5pt}
\item[(ii)]
$\mathcal{F}_g^{\mathsf{SQ}}$ is of degree at most $3g-3$  with respect to $A_2$,
\item[(iii)]
$\frac{\partial^k \mathcal{F}_g^{\mathsf{SQ}}}{\partial T^k}(q) \in \mathbb{C}[L^{\pm1}][A_2,C_1^{-1}]$ for  $g\geq 1$ and  $k\geq 1$,
\vspace{5pt}

\item[(iv)]
${\frac{\partial^k \mathcal{F}_g^{\mathsf{SQ}}}{\partial T^k}}$ is homogeneous of degree $k$
with respect to  $C_1^{-1}$.

\end{enumerate}
\end{Thm}
\noindent Here, $\mathcal{F}^{\mathsf{SQ}}_g=\mathcal{F}^{\mathsf{SQ}}_{g,0}$.

\begin{Thm} \label{ttt} The holomorphic anomaly equations
for the stable quotient invariants of $K\proj^2$
hold for $g\geq 2$:
$$\frac{1}{C_1^2}\frac{\partial \mathcal{F}_g^{\mathsf{SQ}}}{\partial{A_2}}
= \frac{1}{2}\sum_{i=1}^{g-1} 
\frac{\partial \mathcal{F}_{g-i}^{\mathsf{SQ}}}{\partial{T}}
\frac{\partial \mathcal{F}_i^{\mathsf{SQ}}}{\partial{T}}
+
\frac{1}{2}
\frac{\partial^2 \mathcal{F}_{g-1}^{\mathsf{SQ}}}{\partial{T}^2}\,
.$$

\end{Thm}

The derivative of $\mathcal{F}_g^{\mathsf{SQ}}$ (the lift)
with respect to $A_2$ in the holomorphic anomaly
equation of Theorem \ref{ttt} is well-defined since 
$${\mathcal{F}_g^{\mathsf{SQ}}}\in \mathbb{C}[L^{\pm1}][A_2]$$
by Theorem \ref{ooo} part (i).
By Theorem \ref{ooo} parts (ii) and (iii),
$$\frac{\partial \mathcal{F}_{g-i}^{\mathsf{SQ}}}{\partial T}
\frac{\partial \mathcal{F}_{i}^{\mathsf{SQ}}}{\partial T}\, , \
\frac{\partial^2 \mathcal{F}_{g-1}^{\mathsf{SQ}}}{\partial{T}^2}\ 
\in \ \mathbb{C}[L^{\pm1}][A_2,C_1^{-1}]\, $$
are both of degree 2 in $C_1^{-1}$. Hence, 
the holomorphic anomaly equation of Theorem \ref{ttt} may be viewed
as holding 
in $\mathbb{C}[L^{\pm1}][A_2]$
since the factors of $C_1^{-1}$ on the left and right sides cancel.
The
holomorphic anomaly equations here for
$K\PP^2$ are exactly as presented in \cite[(4.27)]{ASYZ} via
B-model physics.

Theorem \ref{ttt} determines
$\mathcal{F}_g^{\mathsf{SQ}}\in \mathbb{C}[L^{\pm1}]$ uniquely as a polynomial
in $A_2$ up to a constant term in $\mathbb{C}[L^{\pm1}]$.
In fact, the degree of the constant term can be bounded
(as will be seen in the proof of Theorem \ref{ttt}). So
Theorem \ref{ttt} determines $\mathcal{F}_g^{\mathsf{SQ}}$ 
from the lower genus theory together with a finite amount of data.


\subsection{Plan of the paper}
After a review of the moduli space of stable quotients 
in Section \ref{sqsqsq} and the associated $\mathsf{T}$-fixed point
loci in Section \ref{locq}, the corresponding $I$-functions and vertex
integrals are discussed in Sections \ref{bcbcbc} and \ref{hgi}.
The localization formula for $K\mathbb{P}^2$ in the
precise form required for the holomorphic anomaly equation
is presented in Sections \ref{hgs} and \ref{svel}. Theorems
\ref{ooo} and \ref{ttt} are proven in Section \ref{hafp}.

At the end of paper, in Section \ref{hakp2}, we prove a 
new holomorphic anomaly
equation for a stable quotient theory of $K\mathbb{P}^2$
{\em with insertions} parallel to (and inspired by)
the recent work  of Oberdieck and Pixton \cite{ObPix} on a cycle
level holomorphic
anomaly equation for the theory of an elliptic curve. 

\subsection{Holomorphic anomaly for the
formal quintic}
Let $g\geq 2$. The Gromov-Witten
invariants of the formal quintic theory were
defined in Section \ref{ham}.
The associated generating series is 
$$\widetilde{\mathcal{F}}^{\mathsf{GW}}_g(Q)\, =\, 
\sum_{d=0}^\infty \widetilde{N}_{g,d}^{\mathsf{GW}} Q^d\, 
\in \, \mathbb{C}[[Q]] \, .$$
We {\em define} the generating series
of stable quotient invariants for
formal quintic theory by the
wall-crossing formula \eqref{345} 
for the true quintic theory,
\begin{equation}\label{jjed}
\widetilde{\mathcal{F}}_{g}^{\mathsf{GW}}(Q(q)) =
I^{\mathsf{Q}}_0(q)^{2g-2} \cdot \widetilde{\mathcal{F}}_{g}^{\mathsf{SQ}}(q)
\end{equation}
with respect to the true quintic mirror map
$$Q(q) = \exp\left(\frac{I^{\mathsf{Q}}_1(q)}
{I^{\mathsf{Q}}_0(q)}\right)\, = \, 
q\cdot \exp\left( \frac{5 \sum_{d=1}^\infty q^d \frac{(5d)!}{(d!)^5} 
\left( \sum_{r=d+1}^{5d} \frac{1}{r}\right)}{\sum_{d=0}^\infty\frac{(5d)!}{(d!)^5}} \right)
\, .$$
Denote the B-model side of \eqref{jjed} by
$$\widetilde{F}^{\mathsf{B}}_g(q)= I^{\mathsf{Q}}_0(q)^{2g-2} \widetilde{\mathcal{F}}^{\mathsf{SQ}}_g(q)\, .$$

In order to state the holomorphic anomaly equations,
we require several series in $q$. First, let
$$L(q) \, =\,   (1-5^5 q)^{-\frac{1}{5}}\,  = \, 1+625q+117185 q^2 +\ldots\, .$$
Let $\mathsf{D}=q\frac{d}{dq}$, and let
$$C_0(q)= I^{\mathsf{Q}}_0\, , \ \ \ 
C_1(q)= \mathsf{D} \left( \frac{I^{\mathsf{Q}}_1}
{I^{\mathsf{Q}}_0}\right)\, ,$$
where $I_0$ and $I_1$ and the hypergeometric series
appearing in the mirror map for the true quintic theory.
We define
\begin{eqnarray*}
K_2(q)&=& -\frac{1}{L^5} \frac{\DD C_0}{C_0}\,  ,
\\
A_2(q)&=& \frac{1}{L^5}\left( -\frac{1}{5}\frac{\DD C_1}{C_1}-\frac{2}{5}\frac{\DD C_0}{C_0}-\frac{3}{25}\right)\, ,\\
A_4(q) &=& \frac{1}{L^{10}} \Bigg(-\frac{1}{25}\left(
\frac{\DD C_0}{C_0}\right)^2-\frac{1}{25}
\left(\frac{\DD C_0}{C_0}\right)\left(\frac{\DD C_1}{C_1}\right)
\, \\
& &
+\frac{1}{25}\DD\left(\frac{\DD C_0}{C_0}\right)+\frac{2}{25^2} \Bigg)\,  ,\\
A_6(q) &=&\frac{1}{31250L^{15} }\Bigg( 4+125 \DD\left(\frac{\DD C_0}{C_0}\right)
+50\left(\frac{\DD C_0}{C_0}\right)
\left(1+10 \DD \left(\frac{\DD C_0}{C_0}
\right)\right)\,    \\
& & -5L^5\Bigg(1+10\left(\frac{\DD C_0}{C_0}\right)+25
\left(\frac{\DD C_0}{C_0}\right)^2+25\DD
\left(\frac{q\frac{d}{dq}C_0}{C_0}\right)\Bigg)\,  \\
& &+125\DD^2\left(\frac{\DD C_0}{C_0}\right)-125\left(\frac{\DD C_0}{C_0}\right)^2\Bigg(\left(\frac{\DD C_1}{C_1}\right)-1\Bigg) \Bigg)\, .
\end{eqnarray*}
Let $T$ be the standard coordinate mirror to $t=\log(q)$,
\begin{align*} T= \frac{I^{\mathsf{Q}}_1(q)}{I^{\mathsf{Q}}_0(q)}\, .
\end{align*}
Then $Q(q)=\exp(T)$ is the mirror map. Let $$\mathbb{C}[L^{\pm1}][A_2,A_4,A_6,C_0^{\pm 1},C_1^{-1},K_2]$$ 
be the free
polynomial ring over $\mathbb{C}[L^{\pm1}]$.

\begin{Thm} \label{ooo5} For the series 
$\widetilde{\mathcal{F}}_g^{\mathsf{B}}$
associated to the formal quintic, 
\begin{enumerate}
\item[(i)]
$\widetilde{\mathcal{F}}_g^{\mathsf{B}} (q) \in \mathbb{C}[L^{\pm1}][A_2,A_4,A_6,C_0^{\pm 1},C_1^{-1},K_2]$
for $g\geq 2$, \vspace{5pt}

\item[(ii)]
$\frac{\partial^k \widetilde{\mathcal{F}}_g^{\mathsf{B}}}{\partial T^k}(q) \in \mathbb{C}[L^{\pm1}][A_2,A_4,A_6,C_0^{\pm 1},C_1^{-1},K_2]$ for  $g\geq 1$, $k\geq 1$,
\vspace{5pt}

\item[(iii)]
${\frac{\partial^k \widetilde{\mathcal{F}}_g^{\mathsf{B}}}{\partial T^k}}$ is homogeneous 
with respect  to $C_1^{-1}$ 
of degree $k$.
\end{enumerate}
\end{Thm}


\begin{Thm} \label{ttt5} The holomorphic anomaly equations
for the series $\widetilde{\mathcal{F}}^{\mathsf{B}}_g$
associated to the formal quintic hold for $g\geq 2$:
\begin{multline*}
\frac{1}{C_0^2C_1^2}\frac{\partial \widetilde{\mathcal{F}}_g^{\mathsf{B}}}{\partial{A_2}}-\frac{1}{5C_0^2C_1^2}\frac{\partial \widetilde{\mathcal{F}}_g^{\mathsf{B}}}{\partial{A_4}}K_2+\frac{1}{50C_0^2C_1^2}\frac{\partial \widetilde{\mathcal{F}}_g^{\mathsf{B}}}{\partial{A_6}}K_2^2
= \\
\frac{1}{2}\sum_{i=1}^{g-1} 
\frac{\partial \widetilde{\mathcal{F}}_{g-i}^{\mathsf{B}}}{\partial{T}}
\frac{\partial \widetilde{\mathcal{F}}_i^{\mathsf{B}}}{\partial{T}}
+
\frac{1}{2}\,
\frac{\partial^2 \widetilde{\mathcal{F}}_{g-1}^{\mathsf{B}}}{\partial{T}^2}\,
.
\end{multline*}
\end{Thm}

\noindent 
The equality of Theorem \ref{ttt5} holds in 
the ring
$$\mathbb{C}[L^{\pm1}][A_2,A_4,A_6,C_0^{\pm1},C_1^{-1},K_2]\, .$$
Theorem \ref{ttt5} exactly matches{\footnote{Our functions
$K_2$ and $A_{2k}$ 
 are normalized differently with respect to $C_0$ and $C_1$.
The dictionary to exactly match the notation of \cite[(2.52)]{ASYZ} is to 
multiply our $K_2$ by $(C_0C_1)^2$ and our $A_{2k}$ by $(C_0C_1)^{2k}$.}}
the
conjectural holomorphic anomaly equation
 \cite[(2.52)]{ASYZ} for
the true quintic theory $\mathcal{F}_g^{\mathsf{SQ}}$.
Theorems \ref{ooo5} and \ref{ttt5}
will be proven in \cite{LP2}.

\subsection{Acknowledgments} 
We are very grateful to  I.~Ciocan-Fontanine, E. Clader, Y. Cooper, F. Janda, B.~Kim, 
A.~Klemm, Y.-P. Lee, A.~Marian, M.~Mari\~no, D.~Maulik, D.~Oprea, Y.~Ruan,  E.~Scheidegger, Y. Toda, and A.~Zinger
for discussions over the years 
about the moduli space of stable quotients and the invariants of Calabi-Yau geometries. 

R.P. was partially supported by 
SNF-200020162928, ERC-2012-AdG-320368-MCSK, SwissMAP, and
the Einstein Stiftung. 
H.L. was supported by the grant ERC-2012-AdG-320368-MCSK.
The results here
were presented 
by H.L. in genus 2 at 
 {\em Curves
on surfaces and 3-folds} at the Bernoulli center
in Lausanne in June 2016
and in all genus at {\em Moduli
of curves, sheaves, and $K3$ surfaces} at Humboldt University
in Berlin in February 2017.


\section{Stable quotients}\label{sqsqsq}

\subsection{Overview}
We review here the basic definitions related to the moduli space of stable quotients following \cite{MOP}. While the
quasimap theory \cite{CKM} is
more general (and the
quasimap terminology
of \cite{CKw,KL} will appear later in
our proofs), only the original moduli spaces of stable quotients are
used in the paper.

\subsection{Stability} 
Let $(C,p_1,\ldots,p_n)$ be an $n$-pointed quasi-stable curve:
\begin{enumerate}
\item[$\bullet$]
$C$ is a reduced, connected, complete, scheme of dimension 1 with   
at worst nodal singularities,
\item[$\bullet$]
 the markings $p_i$ are distinct and lie in the nonsingular
locus $p_i\in C^{ns}$.
\end{enumerate}
Let $q$ be a quotient of the rank $N$ trivial bundle
 $C$,
\begin{equation*}
\com^N \otimes \oh_C \stackrel{q}{\rarr} Q \rarr 0.
\end{equation*}
If the quotient sheaf $Q$ 
is locally free at the nodes and markings of $C$,
 then
$q$ is a {\it quasi-stable quotient}. 
 Quasi-stability of $q$ implies the associated
kernel,
\begin{equation*}
0 \rarr S \rarr
\com^N \otimes \oh_C \stackrel{q}{\rarr} Q \rarr 0,
\end{equation*}
is a locally free sheaf on $C$. Let $r$ 
denote the rank of $S$.

Let $(C,p_1,\ldots,p_n)$ be an $n$-pointed quasi-stable
curve equipped
with a quasi-stable quotient $q$.
The data $(C,p_1,\ldots,p_n,q)$ determine 
a {\it stable quotient} if
the $\mathbb{Q}$-line bundle 
\begin{equation}\label{aam}
\omega_C(p_1+\ldots+p_n)
\otimes (\wedge^{r} S^*)^{\otimes \epsilon}
\end{equation}
is ample 
on $C$ for every strictly positive $\epsilon\in \mathbb{Q}$.
Quotient stability implies
$$2g-2+n \geq 0$$
where $g$ is the arithmetic genus of $C$.

Viewed in concrete terms, no amount of positivity of
$S^*$ can stabilize a genus 0 component 
$$\proj^1\stackrel{\sim}{=}P \subset C$$
unless $P$ contains at least 2 nodes or markings.
If $P$ contains exactly 2 nodes or markings,
then $S^*$ {\it must} have positive degree.

A stable quotient $(C,q)$
yields a rational map from the underlying curve
$C$ to the Grassmannian $\mathbb{G}(r,N)$.
 We be will mainly interested in the projective space case here,
$${\mathbb{G}}(1,m+1)=\proj^{m}\, ,$$
but the definitions are uniform
for all Grassmannian targets.

\subsection{Isomorphism}
Let $(C,p_1,\ldots,p_n)$ be an $n$-pointed curve.
Two quasi-stable quotients
\begin{equation}\label{fpp22}
\com^N \otimes \oh_C \stackrel{q}{\rarr} Q \rarr 0,\ \ \
\com^N \otimes \oh_C \stackrel{q'}{\rarr} Q' \rarr 0
\end{equation}
on $(C,p_1,\ldots,p_n)$ 
are {\it strongly isomorphic} if
the associated kernels 
$$S,S'\subset \com^N \otimes \oh_C$$
are equal.

An {\it isomorphism} of quasi-stable quotients
 $$\phi:(C,p_1,\ldots,p_n,q)\rarr
(C',p'_1,\ldots,p'_n,q')
$$ is
an isomorphism of pointed curves
$$\phi: (C,p_1,\ldots,p_n) \stackrel{\sim}{\rarr} (C',p'_1,\ldots,p'_n)$$
such that
the quotients $q$ and $\phi^*(q')$ 
are strongly isomorphic.
Quasi-stable quotients \eqref{fpp22} on the same
curve $(C,p_1,\ldots,p_n)$
may be isomorphic without being strongly isomorphic.

The following result is proven in \cite{MOP} by
Quot scheme methods from the perspective
of geometry relative to a divisor.

\vspace{8pt}
\noindent{\bf{Theorem A.}}{\it The moduli space of stable quotients 
$\overline{Q}_{g,n}({\mathbb{G}}(r,N),d)$ parameterizing the
data
$$(C,p_1,\ldots,p_n,\  0\rarr S \rarr
\com^N\otimes \oh_C \stackrel{q}{\rarr} Q \rarr 0),$$
with {\it rank}$(S)=r$ and {\it deg}$(S)=-d$,
is a separated and proper Deligne-Mumford stack of finite type
over $\com$.}

\subsection{Structures}\label{strrr}
Over the moduli space of stable quotients, there is a universal
$n$-pointed curve
\begin{equation}\label{ggtt}
\pi: \mathcal{C} \rightarrow \overline{Q}_{g,n}({\mathbb{G}}(r,N),d)
\end{equation}
with a universal quotient
$$0 \rarr \mathsf{S} \rarr \com^N \otimes \oh_{\mathcal{C}} \stackrel{q_U}{\rarr} 
\mathsf{Q}\rarr 0.$$
The subsheaf $\mathsf{S}$ is locally 
free on $\mathcal{C}$ because of the
stability condition.

The moduli space $\overline{Q}_{g,n}({\mathbb{G}}(r,N),d)$ is equipped
with two basic types of maps.
If $2g-2>0$, then the stabilization of $C$
determines a map
$$\nu:\overline{Q}_{g,n}({\mathbb{G}}(r,N),d) \rightarrow \overline{M}_{g,n}$$
by forgetting the quotient.

The general linear group $\mathbf{GL}_N(\com)$ acts on
$\overline{Q}_{g,n}({\mathbb{G}}(r,N),d)$ via 
the standard
action on $\com^N \otimes \oh_C$. The structures
$\pi$, $q_U$,
$\nu$ and the evaluations maps are all $\mathbf{GL}_N(\com)$-equivariant. We will be interested in the
diagonal torus action,
\begin{equation}\label{ggpp2}
\T_{N+1}\subset \mathbf{GL}_N(\mathbb{C})\, .
\end{equation}

\subsection{Obstruction theory}
The moduli of stable
quotients 
 maps 
to the Artin stack of pointed domain curves
$$\nu^A:
\overline{Q}_{g,n}({\mathbb{G}}(r,N),d) \rightarrow {\mathcal{M}}_{g,n}.$$
The moduli  of stable quotients with fixed underlying
pointed curve 
$[C,p_1,\ldots,p_n] \in {\mathcal{M}}_{g,n}$
 is simply
an open set of the Quot scheme of $C$. 
The following result of \cite[Section 3.2]{MOP} is obtained from the
standard deformation theory of the Quot scheme.

\vspace{8pt}
\noindent{\bf{Theorem B.}}{\it The deformation theory of the Quot scheme 
determines a 2-term obstruction theory on the moduli space
$\overline{Q}_{g,n}({\mathbb{G}}(r,N),d)$ relative to
$\nu^A$
given by ${{RHom}}(S,Q)$.}
\vspace{8pt}

More concretely, for the stable quotient,
\begin{equation*}
0 \rarr S \rarr
\com^N \otimes \oh_C \stackrel{q}{\rarr} Q \rarr 0,
\end{equation*} the
deformation and obstruction spaces relative to $\nu^A$ are
$\text{Hom}(S,Q)$ and $\text{Ext}^1(S,Q)$
respectively. Since $S$ is locally free, the higher obstructions
$$\text{Ext}^{k}(S,Q)= H^{k}(C,S^*\otimes Q) = 0, \ \ \ k>1$$
vanish since $C$ is a curve.
An absolute 2-term obstruction theory on the moduli space
$\overline{Q}_{g,n}({\mathbb{G}}(r,N),d)$ is
obtained from Theorem B and the smoothness
of $\mathcal{M}_{g,n}$, see \cite{BF,GP}. 

The $\mathbf{GL}_N(\com)$-action lifts to the
obstruction theory,
and the resulting virtual class is
defined in $\mathbf{GL}_N(\com)$-equivariant cycle theory,
$$[\overline{Q}_{g,n}({\mathbb{G}}(r,N),d)]^{\text{vir}} 
\in A_*^{\mathbf{GL}_N(\com)}
(\overline{Q}_{g,n}({\mathbb{G}}(r,N),d)).$$
Via the restriction to the
torus \eqref{ggpp2}, we obtain
$$[\overline{Q}_{g,n}({\mathbb{G}}(r,N),d)]^{\text{vir}} 
\in A_*^{\T_{N+1}}
(\overline{Q}_{g,n}({\mathbb{G}}(r,N),d)).$$

\section{Localization graphs}

\label{locq}
\subsection{Torus action}
Let $\mathsf{T}=(\com^*)^{m+1}$ act diagonally on the vector space $\mathbb{C}^{m+1}$
with weights
$$-\lambda_0, \ldots, -\lambda_m\, .$$
Denote the $\mathsf{T}$-fixed points of 
the induced $\mathsf{T}$-action on $\proj^m$ by
$$p_0, \ldots, p_{m}\, . $$
The weights of $\mathsf{T}$ on the tangent space $T_{p_j}(\proj^m)$ are
$$\lambda_j-\lambda_0, \ldots, \widehat{\lambda_j-\lambda_j}  ,\ldots, \lambda_j-\lambda_{m}\, .$$

There is an induced $\mathsf{T}$-action on 
the moduli space 
$\overline{Q}_{g,n}(\proj^m,d)$.
The localization formula of \cite{GP} applied to the  virtual fundamental class 
$[\overline{Q}_{g,n}(\proj^m,d)]^{vir}$ will play a fundamental role our paper.
The $\mathsf{T}$-fixed loci are represented in terms of dual graphs,
and the contributions of the $\mathsf{T}$-fixed loci are given by
tautological classes. The formulas here are
standard, see \cite{KL,MOP}.

\subsection{Graphs}\label{grgr}
Let the genus $g$ and the number of markings $n$ for the moduli
space be in
the stable range
\begin{equation}\label{dmdm}
2g-2+n>0\, .
\end{equation}
We can organize the $\mathsf{T}$-fixed loci 
of $\overline{Q}_{g,n}(\proj^m,d)$
according to decorated graphs.
A {\em decorated graph} $\Gamma \in \mathsf{G}_{g,n}(\proj^m)$ consists 
of the data $(\mathsf{V}, \mathsf{E}, 
\mathsf{N}, \mathsf{g}, \mathsf{p} )$ where
\begin{enumerate}
 \item[(i)] $\mathsf{V}$ is the vertex set, 
 \item[(ii)] $\mathsf{E}$ is the edge set (including
 possible self-edges),
 \item[(iii)] $\mathsf{N} : \{1,2,..., n\} \rightarrow \mathsf{V}$ is the marking
 assignment,
   \item[(iv)] $\mathsf{g}: \mathsf{V} \rightarrow \ZZ_{\geq 0}$ is a genus
 assignment satisfying
 $$g=\sum_{v \in V} \mathsf{g}(v)+h^1(\Gamma)\, $$
and for which $(\mathsf{V},\mathsf{E},\mathsf{N},\mathsf{g})$ is stable graph{\footnote
{Corresponding to a stratum of the moduli space
of stable curves $\overline{M}_{g,n}$.}}, 
 \item[(v)] $\mathsf{p} : \mathsf{V} \rightarrow ({\PP ^m})^{\mathsf{T}}$ is an assignment of a $\mathsf{T}$-fixed point $\mathsf{p} (v)$ to each vertex $v \in \mathsf{V}$.
\end{enumerate}
The markings $\mathsf{L}=\{1,\ldots,n\}$ are often called {\em legs}.

To each decorated graph $\Gamma\in \mathsf{G}_{g,n}(\proj^m)$, we associate the set of fixed loci of  
$$\sum_{d\geq 0} \left[\overline{Q}_{g, n} (\proj^m, d)\right]^{\vir} q^d$$
with elements described as follows:
\begin{enumerate}
 
 \item[(a)] If $\{v_{i_1},\ldots,v_{i_k}\}=\{v\, |\, \mathsf{p}(v)=p_i\}$, then $f^{-1}(p_i)$ is a disjoint union of connected stable curves of genera $\mathsf{g}(v_{i_1}),\ldots, \mathsf{g}(v_{i_k})$ and finitely many points.
 
 
 \item[(b)] There is a bijective
  correspondence between the connected components of $C \setminus D$ and the set of edges and legs of $\Gamma$ respecting  vertex incidence where $C$ is domain curve and $D$ is union of all subcurves of $C$ which appear in (a). 
  
\end{enumerate}
We write the localization formula as
$$\sum_{d\geq 0} \left[\overline{Q}_{g, n} (\PP^m, d)\right]^{\vir} q^d =
\sum_{\Gamma\in \mathsf{G}_{g,n}(\PP^m)} \text{Cont}_\Gamma\, .$$
While $\mathsf{G}_{g,n}(\proj^m)$ is a finite set,
each contribution $\text{Cont}_\Gamma$ is
a series in $q$ obtained from
an infinite sum over all edge possibilities (b).

\subsection{Unstable graphs}
The moduli spaces of stable quotients $$\overline{Q}_{0,2}(\proj^m,d) \ \ \
\text{and} \ \ \ \overline{Q}_{1,0}(\proj^m,d)$$
for $d>0$
are the only{\footnote{The moduli spaces
$\overline{Q}_{0,0}(\proj^m,d)$ and
$\overline{Q}_{0,1}(\proj^m,d)$
are empty by the definition of a stable
quotient.}}
cases where
the pair $(g,n)$ does 
{\em not} satisfy the Deligne-Mumford stability condition 
\eqref{dmdm}. 

An appropriate set of decorated graphs $\mathsf{G}_{0,2}(\PP^m)$
 is easily defined: The graphs
$\Gamma \in \mathsf{G}_{0,2}(\PP^m)$ all have 2 vertices
connected by a single edge. Each vertex carries a marking.
All  of the  conditions (i)-(v)
of Section \ref{grgr} are satisfied
except for the stability of $(\mathsf{V},\mathsf{E}, \mathsf{N},\gamma)$.
The  localization formula holds,
\begin{eqnarray}\label{ddgg}
\sum_{d\geq 1} \left[\overline{Q}_{0, 2} (\PP^m, d)\right]^{\vir} q^d &=&
\sum_{\Gamma\in \mathsf{G}_{0,2}(\proj^m)} \text{Cont}_\Gamma\,,
\end{eqnarray}
For $\overline{Q}_{1,0}(\proj^m,d)$, the matter
is more problematic --- usually a marking
is introduced to break the symmetry.

\section{Basic correlators}\label{bcbcbc}
\subsection{Overview}
We review here basic generating series in $q$ which 
arise in  the genus 0 theory of quasimap invariants. The series
will play a fundamental role in
the calculations of Sections \ref{hgi} - \ref{hafp}
related
to the holomorphic anomaly equation for $K\proj^2$.

We fix a torus action $\mathsf{T}=(\CC^*)^3$ on $\PP^2$ with
weights{\footnote{The associated weights on
$H^0(\PP^2,\mathcal{O}_{\PP^2}(1))$ are
$\lambda_0,\lambda_1,\lambda_2$
and so match the conventions of
Section \ref{twth}.}}
$$-\lambda_0, -\lambda_1, -\lambda_2$$
on the vector space $\mathbb{C}^3$.
The $\T$-weight on the fiber over
$p_i$ of the canonical
bundle 
\begin{equation}\label{pqq9}
\mathcal{O}_{\PP^2}(-3) \rightarrow \PP^2
\end{equation}
is $-3\lambda_i$.
The toric Calabi-Yau $K\PP^2$
is the total space of \eqref{pqq9}.

Since the quasimap invariants are independent of $\lambda_i$, we are free
to use the  specialization 
\begin{equation}
\label{spez}
\lambda_1= \zeta \lambda_0\, , \ \ \
    \lambda_2= \zeta^2 \lambda_0
\end{equation}
where $\zeta$ is a primitive third root of unity.
Of course, we then have
\begin{eqnarray*}
    \lambda_0+\lambda_1+\lambda_2&=&0\, , \\
    \lambda_0 \lambda_1+\lambda_1 \lambda_2+\lambda_2 \lambda_0&=&0\, .
\end{eqnarray*}
The specialization \eqref{spez} will be imposed for our {\em entire}
study of $K\PP^2$.

\subsection{First correlators}
We require several correlators defined 
via  the Euler class of the obstruction bundle,
$$e(\text{Obs})= e(R^1\pi_* \mathsf{S}^3)\, ,$$
associated to the $K\PP^2$ geometry
on the moduli space $\overline{Q}_{g,n}(\PP^2,d)$.
The first two are obtained from
standard stable quotient invariants.
For $\gamma_i \in H^*_{\T} (\PP^2)$, let
\begin{eqnarray*}
    \Big \langle \gamma_1\psi^{a_1}, ...,\gamma_n\psi^{a_n} \Big\rangle_{g,n,d}^{\mathsf{SQ}}&=& 
    \int_{[\overline{Q}_{g,n}(\PP^2,d)]^{\vir}} 
    e(\text{Obs})\cdot
    \prod_{i=1}^n \text{ev}_i^*(\gamma_i)\psi_i^{a_i},\\
      \Big\langle \Big\langle \gamma _1\psi  ^{a_1} , ..., \gamma _n\psi  ^{a_n} \Big\rangle \Big\rangle _{0, n}^{\mathsf{SQ}} 
    &=& \sum _{d\geq 0}\sum_{k\geq 0} \frac{q^{d}}{k!}
 \Big\lan    \gamma _1\psi  ^{a_1} , ..., \gamma _n\psi  ^{a_n} , t, ..., t  \Big\ran_{0, n+k, d}^{\mathsf{SQ}} , 
 \end{eqnarray*}
 where, in the second series,
 $t \in H_{\T}^* (\PP^2)$.
 We will systematically use the quasimap notation $0+$
for stable quotients,
\begin{eqnarray*}
    \Big \langle \gamma_1\psi^{a_1}, ...,\gamma_n\psi^{a_n} \Big\rangle_{g,n,d}^{0+}&=&
    \Big \langle \gamma_1\psi^{a_1}, ...,\gamma_n\psi^{a_n} \Big\rangle_{g,n,d}^{\mathsf{SQ}} \\
      \Big\langle \Big\langle \gamma _1\psi  ^{a_1} , ..., \gamma _n\psi  ^{a_n} \Big\rangle \Big\rangle _{0, n}^{0+} 
    &=& 
    \Big\langle \Big\langle \gamma _1\psi  ^{a_1} , ..., \gamma _n\psi  ^{a_n} \Big\rangle \Big\rangle _{0, n}^{\mathsf{SQ}}\, . 
 \end{eqnarray*}

\subsection{Light markings}\label{lightm}
Moduli of quasimaps can be considered with $n$ ordinary (weight 1) markings and $k$ light 
(weight $\epsilon$) markings{\footnote{See Sections 2 and 5 of \cite{BigI}.}},
$$\overline{Q}^{0+,0+}_{g,n|k}(\PP^2,d)\, .$$
Let $\gamma_i \in H^*_{\T} (\PP^2)$ be
equivariant cohomology classes, and
let
$$\delta _j \in H^*_{\T} ([\mathbb{C}^3/\com^* ])$$ 
be classes on the stack quotient. 
Following the notation of \cite{KL}, 
we define series for the $K\PP^2$
geometry,

\begin{multline*}
    \Big\lan \gamma _1\psi  ^{a_1} , \ldots, \gamma _n\psi  ^{a_n} ;  \delta _1, \ldots, \delta _k \Big\ran _{g, n|k, d}^{0+, 0+}  = \\
\int _{[\overline{Q}^{0+, 0+}_{g, n|k} (\PP^2, d)]^{\vir}} 
e(\text{Obs})\cdot
\prod _{i=1}^n \text{ev}_i^*(\gamma _i)\psi _i ^{a_i} 
\cdot \prod _{j=1}^k \widehat{\text{ev}}_j ^* (\delta _j)\, , 
\end{multline*}
\begin{multline*}
\Big \langle \Big\langle \gamma _1\psi  ^{a_1} , \ldots, \gamma _n\psi  ^{a_n} \Big\rangle\Big\rangle _{0, n}^{0+, 0+} 
= \\ \sum _{d\geq 0}\sum_{k\geq 0} \frac{q^{d}}{k!}
 \Big\lan    \gamma _1\psi  ^{a_1} , \ldots, \gamma _n\psi  ^{a_n} ; {t}, \ldots, {t}  
 \Big\ran_{0, n|k, d}^{0+, 0+} \, ,
 \end{multline*}
 where, in the second series,
 ${t} \in H_{\T}^* ([\mathbb{C}^3/\com^* ])$.
 
 For each $\T$-fixed point $p_i\in \PP^2$, let 
 $$e_i= e(T_{p_i}(\PP^2))\cdot(-3\lambda_i)$$
 be the equivariant Euler class of
 the tangent space of $K\PP^2$ at $p_i$. Let
 $$ \phi_i=\frac{-3\lambda_i \prod_{j \ne i}(H-\lambda_j)}{e_i}, \ \ \phi^i=e_i \phi_i\ \ \in H^*_{\T}(\PP^2)\, $$ be cycle classes. Crucial for us are the series
 \begin{align*}
\mathds{S}_i(\gamma) & = e_i \Big\langle \Big\langle  \frac{\phi _i}{z-\psi} , \gamma 
\Big\rangle \Big\rangle _{0, 2}^{0+,0+}\ \ ,  \\
\mathds{V}_{ij}  & =  
\Big\langle \Big\langle  \frac{\phi _i}{x- \psi } ,  \frac{\phi _j}{y - \psi } 
\Big\rangle \Big\rangle  _{0, 2}^{0+,0+}  \ \ . 
 \end{align*}
Unstable degree 0 terms are included by hand in the
above formulas. For $\mathds{S}_i(\gamma)$, the unstable degree 0 term is
$\gamma|_{p_i}$. For $\mathds{V}_{ij}$, the unstable degree 0 term is
$\frac{\delta_{ij}}{e_i(x+y)}$.

 We also write $$\mathds{S}(\gamma)=\sum_{i=0}^2 {\phi_i} \mathds{S}_i(\gamma)\, .$$ 
The series $\mathds{S}_i$ and $\mathds{V}_{ij}$
 satisfy the basic relation
\begin{equation}  \label{wdvv} e_i\mathds{V}_{ij} (x, y)e_j   = 
\frac{\sum _{k=0}^2 \mathds{S}_i (\phi_k)|_{z=x} \, \mathds{S}_j(\phi ^k )|_{z=y}}{x+ y}\,   \end{equation}
 proven{\footnote{In Gromov-Witten
 theory, a parallel relation is
 obtained immediately from the
 WDDV equation and the string equation.
 Since the map forgetting a point
 is not always well-defined for
 quasimaps, a different argument 
 is needed here \cite{CKg}}} in \cite{CKg}.
 
 Associated to each $\T$-fixed point $p_i\in \PP^2$,
 there is a special $\T$-fixed point locus, 
\begin{equation}\label{ppqq}
\overline{Q}^{0+, 0+}_{0, k|m} (\PP^2,d) ^{\T, p_i} \subset
\overline{Q}^{0+, 0+}_{0, k|m}(\PP^2, d)\, ,
\end{equation}
where all markings lie on a single connected
genus 0 domain component contracted to $p_i$.
Let $\text{Nor}$ denote the equivariant
normal bundle 
of $Q^{0+, 0+}_{0, n|k} (\PP^2,d) ^{\T, p_i}$
with respect to the embedding \eqref{ppqq}.
Define 
\begin{multline*}
\Big\lan \gamma _1\psi  ^{a_1} , \ldots, \gamma _n\psi  ^{a_n} ;  \delta _1, ..., \delta _k \Big\ran _{0, n|k, d}^{0+, 0+, p_i}  
=\\
\int _{[\overline{Q}^{0+, 0+}_{0, n|k} (\PP^2, d) ^{\T, p_i}]} 
\frac{e(\text{Obs})}{e(\text{Nor})}\cdot
\prod _{i=1}^n \text{ev}_i^*(\gamma _i)\psi _i ^{a_i} \cdot
\prod _{j=1}^k \widehat{\text{ev}}_j ^* (\delta _j) \, ,
\end{multline*}

\begin{multline*}
  \Big\langle \Big\langle
  \gamma _1\psi  ^{a_1} ,\ldots, \gamma _n\psi  ^{a_n} 
  \Big\rangle \Big\rangle  _{0, n}^{0+, 0+, p_i} =\\
 \sum _{d\geq 0}\sum_{k\geq 0} \frac{q^{d}}{k!}
 \Big\lan    \gamma _1\psi  ^{a_1} , \ldots, \gamma _n\psi  ^{a_n} ; {t}, \ldots, {t}  \Big \ran_{0, n|k, \beta}^{0+, 0+, p_i} \, .
\end{multline*}

\subsection{Graph spaces and I-functions}
\subsubsection{Graph spaces}
The {big I-function} is defined in \cite{BigI}
via the geometry of weighted quasimap graph spaces. 
We briefly summarize the constructions of \cite{BigI}
in the special case of 
%
%
$(0+,0+)$-stability. The more general weightings
discussed in \cite{BigI} will not be
needed here.

As in Section \ref{lightm}, we consider the quotient
$$\com^3/\com^*$$
associated to $\PP^2$.
Following \cite{BigI},
there is a $(0+,0+)$-{\em stable quasimap graph space}
 \begin{equation}\label{xmmx}
     \mathsf{QG}_{g, n|k, d }^{0+,0+} ([\com^3/\com^*] ) \, .
 \end{equation}
A $\CC$-point of the graph space is described by data 
$$((C, {\bf x}, {\bf y}), (f,\varphi):C\lra [\CC^3/\CC^*]\times [\CC^2/\CC^*]).$$ 
By the definition of stability, $\varphi$ is a regular map to $$\PP^1=\CC^2/\!\!/\CC^*\, $$ of class $1$.
Hence, the domain curve $C$ has a distinguished irreducible component $C_0$ canonically isomorphic to $\PP ^1$ via $\varphi$. 
The {\em standard} $\CC ^*$-action, 
\begin{equation}\label{tt44}
t\cdot [\xi _0, \xi _1] = [t\xi _0, \xi _1], \, \, \text{ for } t\in \CC ^*, \, [\xi _0, \xi _1]\in \PP ^1,
\end{equation}
induces  a $\CC ^*$-action on the graph space.

The $\CC^*$-equivariant cohomology of a point is
a free algebra with generator $z$,
$$H^*_{\CC ^*} (\Spec (\CC )) = \QQ [z]\, .$$
Our convention is to define
$z$ as the $\CC^*$-equivariant first Chern class of the tangent line $T_0\PP ^1$ at $0\in\PP^1$ with respect to the
action \eqref{tt44},
$$z=c_1(T_0\PP ^1)\, .$$

The $\T$-action on $\com^3$ lifts to a $\T$-action on
the graph space \eqref{xmmx} which commutes with the
$\CC^*$-action obtained from the distinguished domain component.
As a result, we have a $\T\times \CC^*$-action on the graph space
and
 $\T\times\CC ^*$-equivariant evaluation morphisms
\begin{align}
\nonumber     &\text{ev}_i: \mathsf{QG}_{g, n|k, \beta }^{0+, 0+} ([\CC^3/\CC^*] ) \ra       \PP^2 ,  & i=1,\dots,n\, ,\\
\nonumber &\widehat{\text{ev}}_j:  \mathsf{QG}_{g, n|k, \beta }^{0+,0+} 
([\CC^3/\CC^*] ) \ra       [\CC^3/\CC^*] , & j=1,\dots,k\, .
\end{align}
Since a morphism $$f: C \ra [\CC^3/\CC^*]$$ 
is equivalent to the data of a 
principal $\G$-bundle $P$ on $C$ and a section $u$ of $P\times _{\CC^*} 
\CC^3$, 
there is a natural morphism $$C\ra E\CC^* \times _{\CC^*} \CC^3$$ and hence a pull-back map
 \[ f^*:  H^*_{\CC^*}([\CC^3/\CC^*])  \ra H^*(C). \] 
 The above construction applied
 to the universal curve over the moduli space 
 and the universal morphism to $[\CC^3/\CC^*]$ is $\T$-equivariant.
 Hence,
 we obtain a pull-back map 
 \[ \widehat{\text{ev}}_j^*: H^*_{\T}(\CC^3, \QQ)\otimes_\QQ \QQ[z] \ra H^*_{\T\times \CC^*} (\mathsf{QG}_{g, n|k, \beta }^{0+, 0+} 
 ([\CC^3/\CC^*] ) , \QQ ) \]  
 associated to the evaluation map $\widehat{\text{ev}}_j$.
 


\subsubsection{{\em{I}}-functions}
The description of the fixed loci for the $\CC^*$-action 
on $$\mathsf{QG}_{g, 0|k, d }^{0+,0+} 
([\CC^3/\CC^*] )$$
is parallel to the description in \cite[\S4.1]{CKg0} for the unweighted case. 
In particular, there is a distinguished
subset $\F_{k,d}$ of the $\CC ^*$-fixed locus for which all the markings and the entire curve class $d$ lie  over $0 \in \PP ^1$. The locus
$\F_{k,d}$ comes with
a natural {\it proper} evaluation map $ev_{\bullet}$ obtained
from the generic point of $\PP ^1$:
\[ \text{ev}_\bullet:  \F_{k,d} \ra \CC^3/\!\!/\CC^* =\PP^2 .  \]

We can explicitly write
\begin{equation*}\F_{k,d}\cong \F_d\times 0^k\subset \F_d\times (\PP^1)^k,
\end{equation*}
where $\F_d$ is the $\CC^*$-fixed locus in $\mathsf{QG}^{0+}_{0,0,d}([\CC^3/\CC^*])$ for which the class $d$ is concentrated over $0\in\PP^1$.  The locus $\F_d$ parameterizes
quasimaps of class $d$,
$$f:\PP^1\lra [\CC^3/\CC^*]\, ,$$ with a base-point of 
length $d$ at $0\in\PP^1$. The restriction of $f$ to $\PP^1\setminus\{0\}$ is a constant map to $\PP^2$ defining the evaluation
map $\text{ev}_\bullet$.

As in \cite{CK, CKg0,CKM}, we define the big $\mathds{I}$-function as the generating function for
the push-forward via $ev_\bullet$ of localization residue contributions of $\F_{k,d}$.
For ${\bf t}\in  
H^*_{\T} ([\CC^3/\CC^*], \QQ )\ot _{\QQ} \QQ[z]$, let
 \begin{align*} \mathrm{Res}_{\F_{k,d}}({\bf t}^k) &=
 \prod_{j=1}^k \widehat{\text{ev}}_j^*({\bf t})\, \cap\, \mathrm{Res}_{\F_{k,d}}[
 \mathsf{QG}_{g, 0|k, d }^{0+,0+} 
([\CC^3/\CC^*])
 ]^{\mathrm{vir}} \\
 &=\frac{\prod_{j=1}^k \widehat{\text{ev}}_j^*({\bf t})
 \cap [\F_{k,d}]^{\mathrm{vir}}}
 {\mathrm{e}(\text{Nor}^{\mathrm{vir}}_{\F_{k,d}})},
 \end{align*}
 where 
$\text{Nor}^{\mathrm{vir}}_{\F_{k,d}}$ is the virtual normal bundle.

\begin{Def}\label{Je}
 The big $\mathds{I}$-function for the $(0+,0+)$-stability condition,
 as a formal function in $\bf t$,
 is
\begin{equation*}
\mathds{I}
(q,{\bf t}, z)=\sum_{d\geq 0}\sum_{k\geq 0} \frac{q^d}{k!}
\text{\em ev}_{\bullet\, *}\Big(\mathrm{Res}_{\F_{k,d}}({\bf t}^k)
\Big)\, .
\end{equation*} 
\end{Def}

\subsubsection{Evaluations}


Let $\widetilde{H}\in H^*_\T([\CC^3/\CC^*])$ and $H\in H^*_\T(\PP^2)$
denote the respective hyperplane classes. The $\mathds{I}$-function
of Definition \ref{Je} is evaluated in \cite{BigI}.

\begin{Prop} For ${\bf t}=t\widetilde{H} \in H^*_{\T} ([\CC^3/\CC^*], \QQ)$,
 \begin{align}\label{I_Hyper_P} 
\dsI({t}) = \sum _{d=0}^{\infty} q^d e^{t(H+dz)/z} \frac{ \prod _{k=0}^{3d-1}  (-3H - kz)}{\prod^2_{i=0}\prod _{k=1}^d (H-\lambda_i+kz)} . \end{align}

\end{Prop}

We return now to the functions  $\mathds{S}_i(\gamma)$
defined in Section \ref{lightm}.
Using Birkhoff factorization, an evaluation of
the series $\mathds{S}(H^j)$ can be obtained from the $\dsI$-function, see \cite{KL}:
\begin{align}
\nonumber\mathds{S}({1}) & = \mathds{I} \, , \\
\label{S1}\mathds{S}(H) & = \frac{  z\frac{d}{dt} \mathds{S}({1})}{  z\frac{d}{dt} \mathds{S}({1})|_{t=0,H=1,z=\infty}} \, , \\
\nonumber \mathds{S}(H^2) & = \frac{ z\frac{d}{dt} \mathds{S}(H)}{ z\frac{d}{dt} \mathds{S}(H)|_{t=0,H=1,z=\infty}}\, .
\end{align}
For a  series $F\in \CC[[\frac{1}{z}]]$, the specialization
$F|_{z=\infty}$ denotes constant term of $F$ with respect to $\frac{1}{z}$.

\subsubsection{Further calculations}\label{furcalc}
Define small $I$-function 
$$\overline{\mathds{I}}(q)
\in H^*_{\T}(\PP^2,\QQ)[[q]]$$ by the restriction
\begin{align*}
    \overline{\mathds{I}}(q)
    =\mathds{I}(q,{t})|_{t=0}\, .
\end{align*}
Define differential operators
$$\DD = q\frac{d}{dq}\, , \ \ \ M = H+ z \DD.$$
Applying $z\frac{d}{dt}$ to $\mathds{I}$ and then restricting
to $t=0$ has same effect as applying $M$ to 
$\overline{\mathds{I}}$
 \begin{align*}
     \left[\left(z\frac{d}{dt}\right)^k \mathds{I}\right]\Big|_{t=0} = M^k \, 
     \overline{\mathds{I}}\, .
 \end{align*}
The function 
$\overline{\mathds{I}}$
satisfies following Picard-Fuchs equation
\begin{align}
\label{PF}\Big(M^3-\lambda^3_0+3qM(3M+z)(3M+2z)\Big) 
\overline{\mathds{I}}=0
\end{align}
implied by the Picard-Fuchs equation for $\mathds{I}$,
\begin{align*}
    \left(\left(z\frac{d}{dt}\right)^3-\lambda^3_0+3q\left(z\frac{d}{dt}\right)\left(3\left(z\frac{d}{dt}\right)+z\right)
    \left( 3\left( z\frac{d}{dt}\right)+2z\right)\right)\mathds{I}=0\, .
\end{align*}

The restriction
$\overline{\mathds{I}}|_{H=\lambda_i}$
admits following asymptotic form
\begin{align}
 \label{assym}
 \overline{\mathds{I}}|_{H=\lambda_i}
 = e^{\mu\lambda_i/z}\left( R_0+R_1 \left(\frac{z}{\lambda_i}\right)+R_2 \left(\frac{z}{\lambda_i}\right)^2+\ldots\right)
\end{align}
with series 
$\mu,R_k \in \CC[[q]]$.

A derivation of \eqref{assym} is obtained in \cite{ZaZi} via  
the Picard-Fuchs equation \eqref{PF} for
$\overline{\mathds{I}}|_{H=\lambda_i}$.
The series
$\mu$ and 
$R_k$ are found by solving differential equations obtained from the coefficient of $z^k$. 
For example, 
\begin{eqnarray*}
    1+ \DD\mu&=& L\, , \\
    R_0&=&1\, , \\
    R_1&=& \frac{1}{18}(1-L^2)\, , \\ 
    R_2&=&\frac{1}{648}(1-24L-2L^2+25L^4)\, , 
\end{eqnarray*}
where $L(q) = (1+27q)^{-1/3}$. The specialization \eqref{spez}
is used for these results.

 In \cite{ZaZi}, the authors study the $I$-functions{\footnote{The $I$-function is $\mathcal{F}$ in the notation of
 \cite{ZaZi}.}} related to Calabi-Yau hypersurfaces in $\PP^m$. 
 The local geometry $K\PP^2$ here has a slightly different 
 $I$-function from the function $\mathcal{F}$ in \cite{ZaZi} for the cubic in
 $\PP^2$. More precisely, \begin{align*}
     \mathcal{F}_{-1}(q) = 
     \overline{\mathds{I}}(-q)\, .
 \end{align*} We can either apply the methods of \cite{ZaZi} to our $I$-functions, or we can just take 
 account of the difference between the two $I$-functions (which amounts to a 
 shifting of indices). The latter is easier.

Define the series $C_1$ and $C_2$ by the equations
\begin{align}
C_1 & = z\frac{d}{dt} \mathds{S}({1})|_{z=\infty,t=0,H=1}\, ,  \label{y999}\\
C_2 & = z\frac{d}{dt} \mathds{S}(H)|_{z=\infty,t=0,H=1}\, . \nonumber\\
C_0 & = z\frac{d}{dt} \mathds{S}(H^2)|_{z=\infty,t=0,H=1}\, . \nonumber
\end{align}
The above series are equal (up to sign) to $I_0$, $I_1$ and $I_2$ respectively in the notation of \cite{ZaZi} for the cubic in $\PP^2$.
The following relations were proven in \cite{ZaZi},
\begin{align}\label{c1c2l}
C_0 C_1 C_2 &= (1+27q)^{-1}\, ,\\ \nonumber
C_0 &= C_1\, .
\end{align}


From the equations \eqref{S1} and \eqref{assym}, we can show the series $$\overline{\mathds{S}}_i({1})=\overline{\mathds{S}}({1})|_{H=\lambda_i}\,, \ \ \overline{\mathds{S}}_i(H)=
\overline{\mathds{S}}(H)|_{H=\lambda_i}\, , \ \ \overline{\mathds{S}}_i(H^2)=\overline{\mathds{S}}(H^2)|_{H=\lambda_i}$$ 
have the following asymptotic expansions:
\begin{align}\nonumber
\overline{\mathds{S}}_i({1}) & = e^{\frac{\mu\lambda_i}{z}} \Big(R_{00}+R_{01}\big(\frac{z}{\lambda_i}\big)+R_{02} \big(\frac{z}{\lambda_i}\big)^2+\ldots\Big) \, ,\\  \label{VS}
\overline{\mathds{S}}_i(H) & = e^{\frac{\mu\lambda_i}{z}} \frac{L\lambda_i}{C_1} \Big(R_{10}+R_{11}\big(\frac{z}{\lambda_i}\big)+R_{12} \big(\frac{z}{\lambda_i}\big)^2+\ldots\Big)\, ,\\ \nonumber
\overline{\mathds{S}}_i(H^2) & = e^{\frac{\mu\lambda_i}{z}} \frac{L^2\lambda_i^2}{C_1 C_2} \Big(R_{20}+R_{21}\big(\frac{z}{\lambda_i}\big)+R_{22} \big(\frac{z}{\lambda_i}\big)^2+\ldots)\, .
 \end{align}
We follow here the normalization of \cite{ZaZi}. Note 
\begin{align*}
    R_{0k}=R_{k}.
\end{align*}
As in \cite[Theorem 4]{ZaZi}, we obtain the following  constraints.
 
\begin{Prop}{\em (Zagier-Zinger \cite{ZaZi})}\label{RPoly}
 For all $k\geq 0$, we have
     $$R_k \in \CC[L^{\pm1}]\, .$$
\end{Prop}

Similarly, we also obtain results for $\overline{\mathds{S}}({H})|_{H=\lambda_i}$
and $\overline{\mathds{S}}({H^2})|_{H=\lambda_i}$.
\begin{Lemma}\label{RPoly2} For all $k\geq 0$, we have 
 \begin{align*}
     &R_{1k} \in \CC[L^{\pm1}]\,,\\
     &R_{2k} = Q_{2k} - \frac{R_{1\, k-1}}{L} X\, ,
 \end{align*}
 with $Q_{2k}\in \CC[L^{\pm1}]$ and $X =\frac{\DD C_1}{C_1}$.
\end{Lemma}

\subsection{Determining $\DD X$}

Using Birkhoff factorization of \eqref{S1} further, we have

\begin{align}\label{S3}
 \mathds{S}(H^3) & = \frac{ z\frac{d}{dt} \mathds{S}(H^2)}{ z\frac{d}{dt} \mathds{S}(H^2)|_{t=0,H=1,z=\infty}}.
\end{align}
Since $H^3=\lambda_0^3$ with the specialization \eqref{spez}, we have
\begin{align*}
    \mathds{S}(1)=\frac{1}{\lambda_0^3}\mathds{S}(H^3).
\end{align*}
From \eqref{S1} and \eqref{S3}, we obtain the following
result.

\begin{Lemma}\label{R} We have
 \begin{align*}
    R_{1\,p+1}&=R_{0\,p+1}+\frac{\DD R_{0p}}{L} ,\ \\
    R_{2\,p+1}&=R_{1\,p+1}+\frac{\DD R_{1p}}{L}+\left(\frac{\DD L}{L^2}-\frac{X}{L}\right)R_{1p},\ \\
    R_{0\,p+1}&=R_{2\,p+1}+\frac{\DD R_{2p}}{L}-\left(\frac{\DD L}{L^2}-\frac{X}{L}\right)R_{2p}.
\end{align*}

\end{Lemma}

After setting $p=1$ in Lemma \ref{R}, we find
\begin{equation}\label{drule}
X^2-(L^3-1)X+\DD X-\frac{2}{9}(L^3-1)=0\, .
\end{equation}
 By the above result, the differential ring 
 \begin{equation}\label{ddd333}
 \CC[L^{\pm1}][X,\DD X,\DD\DD X,\ldots]
 \end{equation}
 is just the polynomial ring $\CC[L^{\pm1}][X]$.

We can find the value of $R_{2k}$ in terms of $R_k$ using Lemma \ref{R}.

\begin{align*}
 R_{2\,p+2}=R_{p+2}+2 \frac{\DD R_{p+1}}{L}+\frac{\DD L}{L^2} R_{p+1}+\frac{\DD^2 R_p}{L^2}-\left(\frac{R_{p+1}}{L}+\frac{\DD R_p}{L^2}\right)X.
\end{align*}

\section{Higher genus series on $\overline{M}_{g,n}$}\label{hgi}

 \subsection{Intersection theory on $\overline{M}_{g,n}$} \label{intmg}
 We review here the now standard method used by Givental \cite{Elliptic,SS,Book} to 
 express genus $g$ descendent correlators in terms of genus 0 data.
 
 Let $t_0,t_1,t_2, \ldots$ be formal variables. The series
 $$T(c)=t_0+t_1 c+t_2 c^2+\ldots$$   in the
additional variable $c$ plays a basic role. The variable $c$
will later be  replaced by the first Chern class $\psi_i$ of
 a cotangent line  over $\overline{M}_{g,n}$, 
 $$T(\psi_i)= t_0 + t_1\psi_i+ t_2\psi_i^2 +\ldots\, ,$$
 with the index $i$
 depending on the position of the series $T$ in the correlator.

Let $2g-2+n>0$.
For $a_i\in \mathbb{Z}_{\geq 0}$ and  $\gamma \in H^*(\overline{M}_{g,n})$, define the correlator 
\begin{multline*}
    \lann \psi^{a_1},\ldots,\psi^{a_n}\, | \, \gamma\,  \rann_{g,n}=
    \sum_{k\geq 0} \frac{1}{k!}\int_{\overline{M}_{g,n+k}}
    \gamma \, \psi_1^{a_1}\cdots 
     \psi_n^{a_n}  \prod_{i=1}^k T(\psi_{n+i})\, . 
\end{multline*}
In the above summation,
the $k=0$ term is $$\int_{\overline{M}_{g,n}}\gamma\, \psi_1^{a_1}\cdots\psi_n^{a_n}\,.$$
We also need the following correlator defined for the unstable case,

$$\lan\lan 1,1 \ran\ran_{0,2}=\sum_{k > 0}\frac{1}{k!}\int_{\overline{M}_{0,2+k}}\prod_{i=1}^k T(\psi_{2+i})\,.$$

For formal variables $x_1,\ldots,x_n$, we also define the correlator
\begin{align}\label{derf}
\lannn \frac{1}{x_1-\psi},\ldots,\frac{1}{x_n-\psi}\, \Big| \, \gamma \, \rannn_{g,n}
\end{align}
in the standard way by expanding $\frac{1}{x_i-\psi}$ as a geometric series.

Denote by $\mathds{L}$ the differential operator 
\begin{align*}
        \mathds{L}\, =\, 
        \frac{\partial}{\partial t_0}-\sum_{i=1}^\infty t_i\frac{\partial}{\partial t_{i-1}}
        \, =\, \frac{\partial}{\partial t_0}-t_1\frac{\partial}{\partial t_0}-t_2\frac{\partial}{\partial t_1}-\ldots
        \, .
\end{align*}
 The string equation yields the following result.
 
\begin{Lemma} \label{stst} For $2g-2+n>0$, we have
$\mathds{L}\lann 1,\ldots,1\, | \, \gamma\, \rann_{g,n}=0$ and 
\begin{multline*}
\mathds{L} \lannn \frac{1}{x_1-\psi},\ldots,\frac{1}{x_n-\psi}\, \Big| \,\gamma \, 
\rannn_{g,n}= \\
 \left(\frac{1}{x_1}+\ldots +\frac{1}{x_n}\right)
 \lannn\frac{1}{x_1-\psi},\ldots \frac{1}{x_n-\psi}\, \Big| \, \gamma \, \rannn_{g,n}\, .
 \end{multline*}
\end{Lemma}

After the restriction $t_0=0$ and application of the dilaton equation,
the correlators are expressed in terms of finitely many integrals (by the
dimension constraint). For example,
\begin{eqnarray*}
    \lann 1,1,1\rann_{0,3}\, |_{t_0=0} &= &\frac{1}{1-t_1}\, ,\\
    \lann 1,1,1,1\rann_{0,4}\, |_{t_0=0}& =&\frac{t_2}{(1-t_1)^3}\, ,\\
    \lann 1,1,1,1,1\rann_{0,5}\, |_{t_0=0}&=&\frac{t_3}{(1-t_1)^4}+\frac{3 t_2^2}{(1-t_1)^5}\, ,\\
    \lann 1,1,1,1,1,1\rann_{0,6}\, |_{t_0=0}&=&\frac{t_4}{(1-t_1)^5}+\frac{10 t_2 t_3}{(1-t_1)^6}+\frac{15 t^3_2}{(1-t_1)^7}\, .
\end{eqnarray*}\\

We consider 
$\CC(t_1)[t_2,t_3,...]$
as $\ZZ$-graded ring over $\CC(t_1)$ with 
$$\text{deg}(t_i)=i-1\ \ \text{for $i\geq 2$ .}$$
Define a subspace of homogeneous elements by
$$\CC\left[\frac{1}{1-t_1}\right][t_2,t_3,\ldots]_{\text{Hom}} \subset 
\CC(t_1)[t_2,t_3,...]\, .
$$
We easily see 
$$\lann \psi^{a_1},\ldots,\psi^{a_n}\, | \, \gamma \, \rann_{g,n}\, |_{t_0=0}\ \in\
\CC\left[\frac{1}{1-t_1}\right][t_2,t_3,\ldots]_{\text{Hom}}\, .$$
Using the leading terms (of lowest degree in $\frac{1}{(1-t_1)}$), we obtain the
following result.

\begin{Lemma}\label{basis}
The set of genus 0 correlators
 $$
 \Big\{ \, \lann 1,\ldots,1\rann_{0,n}\, |_{t_0=0} \, \Big\}_{n\geq  4} $$ 
freely generate the ring
 $\CC(t_1)[t_2,t_3,...]$ over $\CC(t_1)$.
\end{Lemma}

By  Lemma \ref{basis}, we can find a unique representation of $\lann \psi^{a_1},\ldots,\psi^{a_n}\rann_{g,n}|_{t_0=0}$
in the  variables
\begin{equation}\label{k3k3}
\Big\{\, \lann 1,\ldots,1\rann_{0,n}|_{t_0=0}\, \Big\}_{n\geq 3}\, .
\end{equation}
The $n=3$ correlator is included in the set \eqref{k3k3} to
capture the variable $t_1$.
For example, in $g=1$,
\begin{eqnarray*}
    \lann 1,1\rann_{1,2}|_{t_0=0}&=&\frac{1}{24}
    \left(\frac{\lann 1,1,1,1,1\rann_{0,5}|_{t_0=0}}{\lan 1,1,1\rann_{0,3}|_{t_0=0}}-\frac{\lann 1,1,1,1\rann^2_{0,4}|_{t_0=0}}{\lann 1,1,1\rann^2_{0,3}|_{t_0=0}}\right)\, ,\\
    \lann 1\rann_{1,1}|_{t_0=0}&=&\frac{1}{24}\frac{\lann 1,1,1,1\rann_{0,4}|_{t_0=0}}{\lann 1,1,1\rann_{0,3}|_{t_0=0}}
    \end{eqnarray*}
A more complicated example in $g=2$ is    
\begin{eqnarray*}
\lann \ \rann_{2,0}|_{t_0=0}&=& \ \ \frac{1}{1152}\frac{\lann 1,1,1,1,1,1\rann_{0,6}|_{t_0=0}}{\lann 1,1,1\rann_{0,3}|_{t_0=0}^2}\\
    & & -\frac{7}{1920}\frac{\lann 1,1,1,1,1\rann_{0,5}|_{t_0=0}\lann 1,1,1,1\rann_{0,4}|_{t_0=0}}{\lann 1,1,1\rann_{0,3}|_{t_0=0}^3}\\& &+\frac{1}{360}\frac{\lann 1,1,1,1\rann_{0,4}|_{t_0=0}^3}{\lann 1,1,1\rann_{0,3}
    |_{t_0=0}^4}\, .
\end{eqnarray*}

\begin{Def} 
For $\gamma \in H^*(\overline{M}_{g,k})$, let $$\pP^{a_1,\ldots,a_n,\gamma}_{g,n}(s_0,s_1,s_2,...)\in \QQ(s_0, s_1,..)$$ be 
the unique rational function satisfying the condition
$$\lann \psi^{a_1},\ldots,\psi^{a_n}\, |\, \gamma\, \rann_{g,n}|_{t_0=0}
=\pP^{a_1,a_2,...,a_n,\gamma}_{g,n}|_{s_i=\lann 1,\ldots,1\rann_{0,i+3}|_{t_0=0}}\, . $$
\end{Def}
 
\begin{Prop}\label{GR1} For $2g-2+n>0$,
we have
 $$\lann 1,\ldots,1\,|\, \gamma\, \rann_{g,n}
=\pP^{0,\ldots,0,\gamma}_{g,n}|_{s_i=\lann 1,\ldots,1\rann_{0,i+3}}\, . $$
\end{Prop} 

\begin{proof}
 Both sides of the equation satisfy the differential equation
 \begin{align*}
     \mathds{L}=0.
 \end{align*}
 By definition, both sides have the same initial conditions at $t_0=0$.
\end{proof}

\begin{Prop}\label{GR2} For $2g-2+n>0$,
 \begin{multline*}
     \lannn \frac{1}{x_1-\psi_1}, \ldots, \frac{1}{x_n-\psi_n}\, \Big| \, \gamma \, \rannn_{g,n}= \\
     e^{\lann 1,1\rann_{0,2}(\sum_i\frac{1}{x_i})}\sum_{a_1,\ldots,a_n}\frac{\pP^{a_1,\ldots,a_n,\gamma}_{g,n}|_{s_i=\lann 1,\ldots,1\rann_{0,i+3}}
     }{x_1^{a_1+1} \cdots x_n^{a_n+1}}.
 \end{multline*}
\end{Prop} 
 
 \begin{proof}
  Both sides of the equation satisfy differential equation
  \begin{align*}
      \mathds{L}-\sum_i\frac{1}{x_i}=0.
  \end{align*}
 Both sides have the same initial conditions at $t_0=0$.
 We use here
     $$\mathds{L} \lann 1,1\rann_{0,2} =1\,, \ \ \ \  \lann 1,1\rann_{0,2}|_{t_0=0}=0\, .$$
 There is no conflict here with Lemma
 \ref{stst} since $(g,n)=(0,2)$ is not
 in the stable range.
 \end{proof}

\subsection{The unstable case $(0,2)$}
The definition given in \eqref{derf}
of the correlator is valid
in the stable range $$2g-2+n>0\, .$$
The unstable case $(g,n)=(0,2)$ plays a
special role. We define
$$\lannn \frac{1}{x_1-\psi_1}, \frac{1}{x_2-\psi_2}\rannn_{0,2}$$
by 
adding the
degenerate term
$$\frac{1}{x_1+x_2}$$
to the terms obtained
by the 
 expansion of $\frac{1}{x_i-\psi_i}$ as 
 a geometric series.
 The degenerate term is associated
to the (unstable) moduli space
of genus 0 with 2 markings.

\begin{Prop}\label{GR22} We have
 \begin{equation*}
     \lannn \frac{1}{x_1-\psi_1}, \frac{1}{x_2-\psi_2} \rannn_{0,2}= 
     e^{\lann 1,1\rann_{0,2}\left(\frac{1}{x_1}+
     \frac{1}{x_2}\right)}\left(\frac{1}{x_1+x_2}\right)\, .
 \end{equation*}
\end{Prop} 
 
 \begin{proof}
  Both sides of the equation satisfy differential equation
  \begin{align*}
      \mathds{L}-\sum_{i=1}^2\frac{1}{x_i}=0.
  \end{align*}
 Both sides have the same initial conditions at $t_0=0$.
 \end{proof}

\subsection{Local invariants and wall crossing} 
The torus $\T$ acts on the moduli spaces
$\overline{M}_{g,n}(\PP^2,d)$  and
$\overline{Q}_{g,n}(\PP^2,d)$.
We consider here special localization contributions 
associated to the fixed points ${p}_i\in \PP^2$.


Consider first the moduli of stable maps.
Let
$$\overline{M}_{g,n}(\PP^2,d)^{\T,p_i}
\subset \overline{M}_{g,n}(\PP^2,d) $$
be the union of
 $\T$-fixed loci which parameterize stable maps
obtained by attaching $\T$-fixed rational tails to a genus $g$, $n$-pointed
Deligne-Mumford stable curve contracted
to the point $p_i\in\PP^2$.
Similarly, let 
$$\overline{Q}_{g,n}(\PP^2,d)^{\T,p_i}\subset
\overline{Q}_{g,n}(\PP^2,d)
$$
be the parallel $\T$-fixed locus
parameterizing stable quotients obtained
by attaching base points
to  a genus $g$, $n$-pointed
Deligne-Mumford stable curve contracted
to the point $p_i\in\PP^2$.

Let $\Lambda_i$ denote the localization of the ring
$$\CC[\lambda^{\pm 1}_0,\lambda^{\pm 1}_1,\lambda^{\pm 1}_2]$$ at 
the three tangent weights at $p_i\in \PP^2$.
Using the virtual
localization formula \cite{GP}, 
there exist unique series
$$S_{p_i}\in\Lambda_i[\psi][[Q]]$$ 
for which the localization contribution 
of the $\T$-fixed locus
$\overline{M}_{g,n}(\PP^2,d)^{\T,p_i}$
to the equivariant Gromov-Witten
invariants of $K\PP^2$
can be written as
\begin{multline*}
    \sum_{d=0}^\infty Q^d \int_{[\overline{M}_{g,n}(K\PP^2,d)^{\T,p_i}]^{\vir}}
    \psi_1^{a_1}\cdots\psi_n^{a_n}=\\
    \sum_{k=0}^\infty \frac{1}{k!} \int_{\overline{M}_{g,n+k}}
    {\mathsf{H}}_{g}^{p_i}\, \psi_1^{a_1}\cdots\psi_n^{a_n}\, \prod_{j=1}^k S_{p_i}(\psi_{n+j})\, .
\end{multline*}
Here, $\mathsf{H}_{g}^{p_i}$ is the standard vertex class, 
\begin{equation}\label{hhbb}
\frac{e(\mathbb{E}_g^*\otimes T_{p_i}(\PP^2))}{e(T_{p_i}(\PP^2))} \cdot \frac{e(\mathbb{E}_g^* \otimes(-3\lambda_i))}{(-3\lambda_i)}\, ,
\end{equation}
obtained the  Hodge bundle $\mathbb{E}_g\rightarrow \overline{M}_{g,n+k}$.

Similarly, the application of the
virtual localization formula to the moduli of stable
quotients yields classes
$$F_{p_i,k}\in H^*(\overline{M}_{g,n|k})\otimes_\CC\Lambda_i$$ 
for which the contribution of $\overline{Q}_{g,n}(\PP^2,d)^{T,p_i}$ is given by
\begin{multline*}
    \sum_{d=0}^\infty q^d \int_{[\overline{Q}_{g,n}(K\PP^2,d)^{\T,p_i}]^{\vir}}\psi_1^{a_1}\cdots
    \psi_n^{a_n}=\\
    \sum_{k=0}^\infty \frac{q^k}{k!} \int_{\overline{M}_{g,n|k}} \mathsf{H}_{g}^{p_i}\, \psi_1^{a_1}\cdots \psi_n^{a_n}\, F_{p_i,k}.
\end{multline*}
Here $\overline{M}_{g,n|k}$ is the moduli space of genus $g$ curves with markings
$$\{p_1,\cdots,p_n\}\cup\{\hat{p}_1\cdots\hat{p}_k \}\in C^{\text{ns}}\subset C$$
satisfying the conditions
\begin{itemize}
 \item[(i)] the points $p_i$ are distinct,
 \item[(ii)] the points $\hat{p}_j$ are distinct from the points $p_i$,
\end{itemize}
with stability given by the ampleness of 
$$\omega_C(\sum_{i=1}^m p_i+\epsilon\sum_{j=1}^k \hat{p}_j)$$
for every strictly positive $\epsilon \in \QQ$.

The Hodge class $\mathsf{H}_{g}^{p_i}$ is given again by
formula \eqref{hhbb} using the Hodge bundle $$\mathbb{E}_g\rightarrow \overline{M}_{g,n|k}\, .$$

\begin{Def}
 For $\gamma\in H^*(\overline{M}_{g,n})$, let
 \begin{eqnarray*}
     \lann \psi_1^{a_1},\ldots,\psi_n^{a_n}\, |\, \gamma\, \rann_{g,n}^{p_i,\infty}
     &=&
     \sum_{k=0}^\infty \frac{1}{k!}
     \int_{\overline{M}_{g,n+k}} \gamma \, \psi_1^{a_1}\cdots \psi_n^{a_n}\prod_{j=1}^k S_{p_i}(\psi_{n+j})\, ,\\
    \lann \psi_1^{a_1},\ldots,\psi_n^{a_n}\, |\, \gamma\, \rann_{g,n}^{p_i,0+}&=&
    \sum_{k=0}^\infty \frac{q^k}{k!} \int_{\overline{M}_{g,n|k}} \gamma \, \psi_1^{a_1}\cdots \psi_n^{a_n}\, F_{p_i,k}\, .
 \end{eqnarray*}
\end{Def}

\
\begin{Prop} [Ciocan-Fontanine, Kim \cite{CKg}] \label{WC} 
For $2g-2+n>0$,
we have the wall crossing relation
$$\lann \psi_1^{a_1},\ldots,\psi_n^{a_n}\, |\, \gamma\, \rann_{g,n}^{p_i,\infty}(Q(q))= \lann \psi_1^{a_1},\ldots,\psi_n^{a_n}\, |\, \gamma\rann_{g,n}^{p_i,0+}(q)$$
 where 
 $Q(q)$ is the mirror map
 $$Q(q)=\exp(I_1^{K\PP^2}(q))\, .$$
\end{Prop}

Proposition \ref{WC} is a consequence
of \cite[Lemma 5.5.1]{CKg}. The mirror
map here is the mirror map for
$K\PP^2$ discussed in Section \ref{holp2}.
 Propositions \ref{GR1} and \ref{WC} together yield 
 \begin{eqnarray*}
 \lann 1,\ldots,1\,  |\, \gamma \, \rann_{g,n}^{p_i,\infty}& =& \pP^{0,\ldots,0,\gamma}_{g,n}\big(\lann 1,1,1\rann_{0,3}^{p_i,\infty},\lann 1,1,1,1\rann _{0,4}^{p_i,\infty},\ldots\big)\, ,\\
 \lann 1,\ldots,1\,  |\, \gamma \, \rann_{g,n}^{p_i,0+}&=&\pP^{0,\ldots,0,\gamma}_{g,n}\big(\lann
 1,1,1\rann_{0,3}^{p_i,0+},\lann 1,1,1,1\rann _{0,4}^{p_i,0+},\ldots\big)\, .
 \end{eqnarray*}
Similarly, using Propositions \ref{GR2} and \ref{WC}, we obtain
\begin{multline*}
\lannn \frac{1}{x_1-\psi}, \ldots, \frac{1}{x_n-\psi}\, \Big| \, \gamma \, \rannn_{g,n}^{p_i,\infty}= \\
     e^{\lann 1,1\rann^{p_i,\infty}_{0,2}\left(\sum_i\frac{1}{x_i}\right)}\sum_{a_1,\ldots,a_n}\frac{\pP^{a_1,\ldots,a_n,\gamma}_{g,n}\big(\lann 1,1,1\rann_{0,3}^{p_i,\infty},\lann 1,1,1,1\rann_{0,4}^{p_i,\infty},\ldots \big)}{x_1^{a_1+1}\cdots x_n^{a_n+1}}\, ,
\end{multline*}
\begin{multline}\label{ppqqpp}
\lannn \frac{1}{x_1-\psi}, \ldots, \frac{1}{x_n-\psi}\, \Big| \, \gamma \, \rannn_{g,n}^{p_i,0+}= \\
     e^{\lann 1,1\rann^{p_i,0+}_{0,2}\left(\sum_i\frac{1}{x_i}\right)}\sum_{a_1,\ldots,a_n}\frac{\pP^{a_1,\ldots,a_n,\gamma}_{g,n}\big(\lann 1,1,1\rann_{0,3}^{p_i,0+},\lann 1,1,1,1\rann_{0,4}^{p_i,0+},\ldots \big)}{x_1^{a_1+1}\cdots x_n^{a_n+1}}\, .
\end{multline}


\section{Higher genus series on $K\PP^2$}\label{hgs}
\subsection{Overview}
We apply the localization strategy introduced first by Givental  \cite{Elliptic,SS,Book} for Gromov-Witten theory to the stable quotient invariants of local $\PP^2$. 
The contribution $\text{Cont}_\Gamma(q)$ 
discussed in Section \ref{locq} 
of a graph $\Gamma \in \mathsf{G}_{g}(\PP^2)$ 
can be separated into vertex and edge contributions.
We express the vertex and edge contributions in terms of
the series $\mathds{S}_i$ and $\mathds{V}_{ij}$ of Section \ref{lightm}.

\subsection{Edge terms}
Recall the definition{\footnote{We use
the variables $x_1$ and $x_2$ here instead
of $x$ and $y$.}}of $\mathds{V}_{ij}$
given in Section \ref{lightm},
\begin{equation}\label{dfdf6}
\mathds{V}_{ij}  =  
\Big\langle \Big\langle  \frac{\phi _i}{x- \psi } ,  \frac{\phi _j}{y - \psi } 
\Big\rangle \Big\rangle  _{0, 2}^{0+,0+}  \, .
\end{equation}
Let $\overline{\mathds{V}}_{ij}$ denote
the restriction of $\mathds{V}_{ij}$
to $t=0$.
Via formula \eqref{ddgg},
$\overline{\mathds{V}}_{ij}$ is a summation of contributions of fixed loci indexed by
a graph $\Gamma$ consisting of two vertices 
connected by a unique edge. 
Let $w_1$ and $w_2$ be 
$\T$-weights. Denote by $$\overline{{\mathds{V}}}_{ij}^{w_1,w_2}$$ the summation of contributions of $\T$-fixed loci with
tangent weights precisely $w_1$
and $w_2$ on the first rational components
which exit the vertex components over
$p_i$ and $p_j$.

The series $\overline{{\mathds{V}}}_{ij}^{w_1,w_2}$
includes {\em both} vertex and edge
contributions.
By definition \eqref{dfdf6} and the virtual localization formula, we find the
following relationship between
$\overline{\mathds{V}}_{ij}^{w_1,w_2}$
and the corresponding
pure edge contribution $\mathsf{E}_{ij}^{w_1,w_2}$,

\begin{eqnarray*}
    e_i\overline{\mathds{V}}_{ij}^{w_1,w_2}e_j
    &=& \lannn \frac{1}{w_1-\psi},\frac{1}{x_1-\psi}\rannn^{p_i,0+}_{0,2}\mathsf{E}_{ij}^{w_1,w_2}
    \lannn \frac{1}{w_2-\psi},\frac{1}{x_2-\psi}\rannn^{p_j,0+}_{0,2}\\
    &=&\frac{e^{\frac{\lann 1,1\rann^{p_i,0+}_{0,2}}{w_1}+\frac{\lann 1,1\rann^{p_j,0+}_{0,2}}{x_1}}}{w_1+x_1}
    \, \mathsf{E}^{w_1,w_2}_{ij}\, \frac{e^{\frac{\lann 1,1\rann^{p_i,0+}_{0,2}}{w_2}+\frac{\lann 1,1\rann^{p_j,0+}_{0,2}}{x_2}}}{w_2+x_2}
     \end{eqnarray*}
    
\begin{align*}        
    =\sum_{a_1,a_2}e^{\frac{\lann 1,1\rann^{p_i,0+}_{0,2}}{x_1}+\frac{\lann 1,1\rann^{p_i,0+}_{0,2}}{w_1}}e^{\frac{\lann 1,1\rann^{p_j,0+}_{0,2}}{x_2}+\frac{\lann 1,1\rann^{p_j,0+}_{0,2}}{w_2}}(-1)^{a_1+a_2} \frac{
    \mathsf{E}^{w_1,w_2}_{ij}}{w_1^{a_1}w_2^{a_2}}x_1^{a_1-1}x_2^{a_2-1}\, .
\end{align*}
After summing over all possible weights, we obtain
$$
    e_i\left(\overline{\mathds{V}}_{ij}-\frac{\delta_{ij}}{e_i(x+y)}\right)e_j=\sum_{w_1,w_2} e_i\overline{\mathds{V}}_{ij}^{w_1,w_2}e_j\, .$$
The above calculations immediately yield
the following result.
    

\begin{Lemma}\label{Edge} We have
 \begin{multline*}
 \left[e^{-\frac{\lann1,1\rann^{p_i,0+}_{0,2}}{x_1}}
       e^{-\frac{\lann1,1\rann^{p_j,0+}_{0,2}}{x_2}}e_i\left(\overline{\mathds{V}}_{ij}-\frac{\delta_{ij}}{e_i(x+y)}\right)e_j\right]_{x_1^{a_1-1}x_2^{a_2-1}}=\\
       \sum_{w_1,w_2}
       e^{\frac{\lann1,1\rann^{p_i,0+}_{0,2}}{w_1}}e^{\frac{\lann1,1\rann^{p_j,0+}_{0,2}}{w_2}}(-1)^{a_1+a_2}\frac{\mathsf{E}_{ij}^{w_1,w_2}}{w_1^{a_1}w_2^{a_2}}\, .
 \end{multline*}
\end{Lemma}

\noindent The notation $[\ldots]_{x_1^{a_1-1}x_2^{a_2-1}}$
in Lemma \ref{Edge} denotes the coefficient of
 $x_1^{a_1-1}x_2^{a_2-1}$ in the series expansion
 of the argument.

\subsection{A simple graph}\label{simgr}
Before treating the general case, we present
the localization formula for a simple graph{\footnote{We
follow here the notation of Section \ref{locq}.}.
Let
$\Gamma\in \mathsf{G}_{g}(\PP^2)$ 
consist of two vertices   and one edge,
$$v_1,v_2\in \Gamma(V)\, , \ \ \ \ 
e\in \Gamma(E)\, $$
with genus and $\T$-fixed point assignments
$$\mathsf{g}(v_i)=g_i\, , \ \ \ \ \mathsf{p}(v_i)=p_i\, .$$

Let $w_1$ and $w_2$ be tangent
weights at the vertices $p_1$ and $p_2$
respectively. Denote by $\text{Cont}_{\Gamma,w_1,w_2}$
the summation of contributions to
\begin{equation}\label{zlzl}
\sum_{d>0} q^d\, \left[\overline{Q}_{g}(K\PP^2,d)\right]^{\vir}
\end{equation}
of $\T$-fixed loci with
tangent weights precisely $w_1$
and $w_2$ on the first rational components
which exit the vertex components over
$p_1$ and $p_2$.
We can express the localization formula for 
\eqref{zlzl} as
$$
\lannn \frac{1}{w_1-\psi}\, \Big|\, \mathsf{H}_{g_1}^{p_1}
\rannn_{g_1,1}^{p_1,0+}
\mathsf{E}^{w_1,w_2}_{12} \lannn\frac{1}{w_2-\psi}\, \Big|\, \mathsf{H}_{g_2}^{p_2}
\rannn_{g_2,1}^{p_2,0+} $$
which equals
$$\sum_{a_1,a_2} e^{\frac{\lann1,1\rann^{p_1,0+}_{0,2}}{w_1}}\frac{\ppl
{\psi^{a_1-1}} \, \Big|\, \mathsf{H}_{g_1}^{p_1}     \ppr_{g_1,1}^{p_1,0+}} {w_1^{a_1}} \mathsf{E}^{w_1,w_2}_{12} e^{\frac{\lann 1,1\rann^{p_2,0+}_{0,2}}{w_2}}\frac{\ppl {\psi^{a_2-1}} \, \Big|\, \mathsf{H}_{g_2}^{p_2}\ppr_{g_2,1}^{p_2,0+}}{w_2^{a_2}}
$$
where $\mathsf{H}_{g_i}^{p_i}$ is
the Hodge class \eqref{hhbb}. We have used here
the notation
\begin{multline*}
\ppl
\psi^{k_1}_1, \ldots,\psi^{k_n}_n \, \Big|\, \mathsf{H}_{h}^{p_i}     \ppr_{h,n}^{p_i,0+} 
=\\
\pP^{k_1,\ldots,k_n,\mathsf{H}_{h}^{p_i}  }_{h,1}\big(\lann 1,1,1\rann_{0,3}^{p_i,0+},\lann 1,1,1,1\rann_{0,4}^{p_i,0+},\ldots \big)
\,
\end{multline*}
and applied \eqref{ppqqpp}.

After summing over all possible weights $w_1,w_2$ and
applying 
Lemma \ref{Edge}, we obtain the following result for the full contribution $$\text{Cont}_\Gamma = \sum_{w_1,w_2} \text{Cont}_{\Gamma,w_1,w_2}$$
of $\Gamma$ to $\sum_{d\geq 0} q^d \left[ \overline{Q}_{g}(K\PP^2,d)\right]^{\vir}$.

\begin{Prop} We have \label{propsim}
 \begin{multline*}
     \text{\em Cont}_{\Gamma}=
     \sum_{a_1,a_2>0}
     \ppl
{\psi^{a_1-1}}  \, \Big|\, \mathsf{H}_{g_1}^{p_i}\,     \ppr_{g_1,1}^{p_i,0+}
\ppl
{\psi^{a_2-1}}  \, \Big|\, \mathsf{H}_{g_2}^{p_j}\,     \ppr_{g_2,1}^{p_j,0+}\ \ \ \ \ \ \ \ \ \ \ \\
\ \ \ \ \ \ \ \ \ \ \cdot
     (-1)^{a_1+a_2}\left[e^{-\frac{\lann1,1\rann^{p_i,0+}_{0,2}}{x_1}}
       e^{-\frac{\lann1,1\rann^{p_j,0+}_{0,2}}{x_2}}e_i\left(\overline{\mathds{V}}_{ij}-\frac{\delta_{ij}}{e_i(x_1+x_2)}\right)e_j\right]_{x_1^{a_1-1}x_2^{a_2-1}}\, .
 \end{multline*}
\end{Prop}

\subsection{A general graph} We apply the argument of Section \ref{simgr}
to obtain a contribution formula for a general graph $\Gamma$.

Let $\Gamma\in \mathsf{G}_{g,0}(\PP^2)$ be a decorated graph as defined in Section \ref{locq}. The {\em flags} of $\Gamma$ are the 
half-edges{\footnote{Flags are either half-edges or markings.}}. Let $\mathsf{F}$ be the set of flags. 
Let
$$\mathsf{w}: \mathsf{F} \rightarrow \text{Hom}(\T, \com^*)\otimes_{\mathbb{Z}}{\mathbb{Q}}$$
be a fixed assignment of $\T$-weights to each flag.

We first consider the contribution $\text{Cont}_{\Gamma,\mathsf{w}}$ to 
$$\sum_{d\geq 0} q^d \left[\overline{Q}_g(K\PP^2,d)\right]^{\text{vir}}$$
of the $\T$-fixed loci associated $\Gamma$ satisfying
the following property:
the tangent weight on
the first rational component corresponding
to each $f\in \mathsf{F}$ is
exactly given by $\mathsf{w}(f)$.
We have 
\begin{equation}
    \label{s234}
    \text{Cont}_{\Gamma,\mathsf{w}} = \frac{1}{|\text{Aut}(\Gamma)|}
    \sum_{\mathsf{A} \in \ZZ_{> 0}^{\mathsf{F}}} \prod_{v\in \mathsf{V}} \text{Cont}^{\mathsf{A}}_{\Gamma,\mathsf{w}} (v)\prod _{e\in \mathsf{E}} {\text{Cont}}_{\Gamma,\mathsf{w}}(e)\, .
\end{equation}
 The terms on the  right side of \eqref{s234} 
require definition:
\begin{enumerate}
\item[$\bullet$] The sum on the right is over 
the set $\ZZ_{> 0}^{\mathsf{F}}$ of
all maps 
$$\mathsf{A}: \mathsf{F} \rightarrow \ZZ_{> 0}$$
corresponding to the sum over $a_1,a_2$ in
Proposition \ref{propsim}.
\item[$\bullet$]
For $v\in \mathsf{V}$ with 
$n$ incident
flags with $\mathsf{w}$-values $(w_1,\ldots,w_n)$ and
$\mathsf{A}$-values
$(a_1,a_2,...,a_n)$, 
\begin{align*}
    \text{Cont}^{\mathsf{A}}_{\Gamma,{\mathsf{w}}}(v)=
    \frac{\ppl
\psi_1^{a_1-1}, \ldots,
\psi_n^{a_n-1}
\, \Big|\, \mathsf{H}_{\mathsf{g}(v)}^{\mathsf{p}(v)}\,     \ppr_{\mathsf{g}(v),n}^{\mathsf{p}(v),0+}}
{w_1^{a_1} \cdots w_n^{a_n}}\, .
\end{align*}
\item[$\bullet$]
For $e\in \mathsf{E}$ with 
assignments $(\mathsf{p}(v_1), \mathsf{p}(v_2))$
for the two associated vertices{\footnote{In case $e$
is self-edge, $v_1=v_2$.}} and 
$\mathsf{w}$-values $(w_1,w_2)$ for the two associated flags,
    $$    
    \text{Cont}_{\Gamma,\mathsf{w}}(e)=
    e^{\frac{\lann1,1\rann^{\mathsf{p}(v_1),0+}_{0,2}}{w_1}}
    e^{\frac{\lann1,1\rann^{\mathsf{p}(v_2),0+}_{0,2}}{w_2}}
    \mathsf{E}^{w_1,w_2}_{\mathsf{p}(v_1),\mathsf{p}(v_2)}\, .$$
\end{enumerate}
The localization formula then yields \eqref{s234}
just as in the simple case of Section \ref{simgr}.

By summing the contribution \eqref{s234} of $\Gamma$ over
all the weight functions $\mathsf{w}$
and applying Lemma \ref{Edge}, we obtain
the following result which generalizes 
Proposition \ref{propsim}.

\begin{Prop}\label{VE} We have
 $$
 \text{\em Cont}_\Gamma
     =\frac{1}{|\text{\em Aut}(\Gamma)|}
     \sum_{\mathsf{A} \in \ZZ_{> 0}^{\mathsf{F}}} \prod_{v\in \mathsf{V}} 
     \text{\em Cont}^{\mathsf{A}}_\Gamma (v)
     \prod_{e\in \mathsf{E}} \text{\em Cont}^{\mathsf{A}}_\Gamma(e)\, ,
 $$
 where the vertex and edge contributions 
 with incident flag $\mathsf{A}$-values $(a_1,\ldots,a_n)$
 and $(b_1,b_2)$ respectively are
 \begin{eqnarray*}
    \text{\em Cont}^{\mathsf{A}}_\Gamma (v)&=&
    \ppl
\psi_1^{a_1-1}, \ldots,
\psi_n^{a_n-1}
\, \Big|\, \mathsf{H}_{\mathsf{g}(v)}^{\mathsf{p}(v)}\,
  \ppr_{\mathsf{g}(v),n}^{\mathsf{p}(v),0+}\,  ,\\
    \text{\em Cont}^{\mathsf{A}}_\Gamma(e)
    &=&
    (-1)^{b_1+b_2}\left[e^{-\frac{\lann1,1\rann^{\mathsf{p}(v_1),0+}_{0,2}}{x_1}}
       e^{-\frac{\lann1,1\rann^{\mathsf{p}(v_2),0+}_{0,2}}{x_2}}e_i\left(\overline{\mathds{V}}_{ij}-\frac{1}{e_i(x+y)}\right)e_j\right]_{x_1^{b_1-1}x_2^{b_2-1}}\, ,
 \end{eqnarray*}
where $\mathsf{p}(v_1)=p_i$ and $\mathsf{p}(v_2)=p_j$ in the second equation. 
\end{Prop}

\subsection{Legs} 
Let $\Gamma \in \mathsf{G}_{g,n}(\PP^2)$ be a decorated graph
with markings. While no markings are needed to define the
stable quotient invariants of $K\PP^2$, the contributions
of decorated graphs with markings will appear in the
proof of the holomorphic anomaly equation.
The formula for the contribution $\text{Cont}_\Gamma(H^{k_1},\ldots,H^{k_n})$
of $\Gamma$ to 
\begin{align*}
    \sum_{d\ge 0}q^d \prod_{j=0}^n \text{ev}^*(H^{k_j})\cap\left[ \overline{Q}_{g,n}(K\PP^2,d)\right]^{\vir} 
\end{align*}
is given by the following result.
\begin{Prop}\label{VEL} We have
 \begin{multline*}
 \text{\em Cont}_\Gamma(H^{k_1},\ldots,H^{k_n})
     =\\\frac{1}{|\text{\em Aut}(\Gamma)|}
     \sum_{\mathsf{A} \in \ZZ_{>0}^{\mathsf{F}}} \prod_{v\in \mathsf{V}} 
     \text{\em Cont}^{\mathsf{A}}_\Gamma (v)
     \prod_{e\in \mathsf{E}} \text{\em Cont}^{\mathsf{A}}_\Gamma(e)
     \prod_{l\in \mathsf{L}} \text{\em Cont}^{\mathsf{A}}_\Gamma(l)\, ,
 \end{multline*}
 where the leg contribution 
 is 
 \begin{eqnarray*}
     \text{\em Cont}^{\mathsf{A}}_\Gamma(l)
    &=&
    (-1)^{\mathsf{A}(l)-1}\left[e^{-\frac{\lann1,1\rann^{\mathsf{p}(l),0+}_{0,2}}{z}}
       \overline{\mathds{S}}_{\mathsf{p}(l)}(H^{k_l})\right]_{z^{\mathsf{A}(l)-1}}\, .
 \end{eqnarray*}
The vertex and edge contributions are same as before.
\end{Prop}

The proof of Proposition \ref{VEL} 
follows the vertex and edge analysis. We leave the
details as an exercise for the reader.
The parallel statement for Gromov-Witten theory
can be found in \cite{Elliptic, SS,Book}.

\section{Vertices, edges, and legs} \label{svel}
\subsection{Overview}
Using the results of Givental \cite{Elliptic,SS,Book} combined with wall-crossing  \cite{CKg}, we calculate here 
the vertex and edge contributions
 in terms of the function $R_k$ of Section \ref{furcalc}.

\subsection{Calculations in genus 0}
We follow the notation introduced in Section \ref{intmg}. Recall
the series
$$T(c)=t_0 +t_1 c+t_2 c^2+\ldots\, .$$ 

\begin{Prop} {\em (Givental \cite{Elliptic,SS,Book})} For $n\geq 3$, we have
\begin{multline*}
\lann
 1,\ldots,1\rann_{0,n}^{p_i,\infty} = \\
 \left(\sum_{k\geq 0}\frac{1}{k!}\int_{\overline{M}_{0,n+k}}T(\psi_{n+1})\cdots T(\psi_{n+k})\right)\Big|_{t_0=0,t_1=0,t_{j\ge 2}=(-1)^j\frac{Q_{j-1}}{\lambda_i^{j-1}}}
 \end{multline*}
 where the functions $Q_l$ are defined by 
 \begin{align*}
     \overline{\mathds{S}}^{\infty}_i(1) =  e_i \Big\langle \Big\langle  \frac{\phi _i}{z-\psi} , 1 
\Big\rangle \Big\rangle _{0, 2}^{p_i,\infty}=e^{\frac{\lann1,1\rann^{p_i,\infty}_{0,2}}{z}}
\left( \sum_{l=0}^\infty Q_l \left(\frac{z}{\lambda_i}\right)^{l}\right)\, .
 \end{align*}
\end{Prop}

The existence of the above asymptotic expansion of $\overline{\mathds{S}}^\infty_i(1)$ can also be proven by the argument of  \cite[Theorem 5.4.1]{CKg0}.
Similarly, we have an asymptotic expansion of $\overline{\mathds{S}}_i(1)$,
\begin{align*}
    \overline{\mathds{S}}_i(1) =e^{\frac{\lann1,1\rann^{p_i,0+}_{0,2}}{z}}
\left( \sum_{l=0}^\infty R_l \left(\frac{z}{\lambda_i}\right)^{l}\right)\, .
\end{align*}
By \eqref{VS}, we have
\begin{align*}
    \lann1,1\rann^{p_i,0+}_{0,2} = \mu \lambda_i.
\end{align*}

After applying the wall-crossing result of Proposition \ref{WC}, we \
obtain
\begin{eqnarray*}
    \lann
 1,\ldots,1\rann_{0,n}^{p_i,\infty}(Q(q)) &=&\lann
 1,\ldots,1\rann_{0,n}^{p_i,0+}(q), \\
 \overline{\mathds{S}}^{\infty}_i(1)(Q(q))&=&\overline{\mathds{S}}_i(1)(q),
\end{eqnarray*}
where $Q(q)$ is mirror map for $K\PP^2$ as before. By comparing asymptotic expansions of $\overline{\mathds{S}}^{\infty}_i(1)$ and $\overline{\mathds{S}}_i(1)$, we get
a wall-crossing relation between $Q_l$ and $R_l$,
\begin{align*}
 Q_l(Q(q))=R_l(q)\, .
\end{align*}
We have proven the following result.

\begin{Prop}\label{q2q2} For $n\geq 3$, we have \label{zaa3}
\begin{multline*}
\lann
 1,\ldots,1\rann_{0,n}^{p_i,0+} = \\
 \left(\sum_{k\geq 0}\frac{1}{k!}\int_{\overline{M}_{0,n+k}}T(\psi_{n+1})\cdots T(\psi_{n+k})\right)\Big|_{t_0=0,t_1=0,t_{j\ge 2}=(-1)^j\frac{R_{j-1}}{\lambda_i^{j-1}}}\, .
 \end{multline*}
\end{Prop}

 Proposition \ref{q2q2} immediately implies the evaluation
\begin{equation} \label{fxxf}
\lann
 1,1,1\rann_{0,3}^{p_i,0+}=1\, .
 \end{equation}
Another simple consequence of Proposition \ref{zaa3} is the following 
 basic property.
\begin{Cor}\label{Poly} For $n\geq 3$, we have
 $
 \lann
 1,\ldots,1\rann_{0,n}^{p_i,0+} \in \CC[R_1,R_2,...][\lambda_i^{-1}]
 $.
\end{Cor}

\subsection{Vertex and edge analysis}
By Proposition \ref{VE}, we have decomposition of the
contribution to $\Gamma\in \mathsf{G}_{g}(\PP^2)$ to
the stable quotient theory of 
$K\PP^2$ 
into vertex terms and edge terms
$$
 \text{Cont}_\Gamma
     =\frac{1}{|\text{Aut}(\Gamma)|}
     \sum_{\mathsf{A} \in \ZZ_{> 0}^{\mathsf{F}}} \prod_{v\in \mathsf{V}} 
     \text{Cont}^{\mathsf{A}}_\Gamma (v)
     \prod_{e\in \mathsf{E}} \text{Cont}^{\mathsf{A}}_\Gamma(e)\, .
 $$


\begin{Lemma}\label{L1} We have
    $\text{\em Cont}^{\mathsf{A}}_\Gamma (v)\in \CC(\lambda_0,\lambda_1,\lambda_2)[L^{\pm1}]$. 
\end{Lemma}

\begin{proof} By Proposition \ref{VE}, 
$$\text{Cont}^{\mathsf{A}}_\Gamma (v) = 
    \ppl
\psi_1^{a_1-1}, \ldots,
\psi_n^{a_n-1}
\, \Big|\, \mathsf{H}_{\mathsf{g}(v)}^{\mathsf{p}(v)}\,
  \ppr_{\mathsf{g}(v),n}^{\mathsf{p}(v),0+}\, .$$
 The right side of the above
  formula is a polynomial 
  in the variables
$$\frac{1}{\lann 1,1,1\rann^{\mathsf{p}(v),0+}_{0,3}}\ \ \ 
\text{and} \ \ \
 \Big\{ \, \lann 1,\ldots,1\rann^{\mathsf{p}(v),0+}_{0,n}\, |_{t_0=0} \, \Big\}_{n\geq  4}\, $$
 with coefficients in $\mathbb{C}(\lambda_0,\lambda_1,\lambda_2)$.
The Lemma then follows from 
the evaluation \eqref{fxxf}, Corollary \ref{Poly},
and 
Proposition \ref{RPoly}. 
\end{proof}

Let $e\in \mathsf{E}$ be an edge connecting the $\T$-fixed points $p_i, p_j \in \PP^2$. Let
the $\mathsf{A}$-values of the respective
half-edges be $(k,l)$.

\begin{Lemma}\label{L2} We have
 $\text{\em Cont}^{\mathsf{A}}_\Gamma(e) \in \CC(\lambda_0,\lambda_1,\lambda_2)[L^{\pm1},X]$ and
 \begin{enumerate}
 \item[$\bullet$]
 the degree of $\text{\em Cont}^{\mathsf{A}}_\Gamma(e)$ with respect to $X$ is $1$,
 \item[$\bullet$]
 the coefficient of $X$ in 
 $\text{\em Cont}^{\mathsf{A}}_\Gamma(e)$
 is 
 $$(-1)^{k+l}\frac{3 R_{1\,k-1} R_{1\,l-1}}{L \lambda_i^{k-2} \lambda_j^{l-2}}\, .$$ 
 \end{enumerate}
\end{Lemma}

\begin{proof}
 By Proposition \ref{VE}, 
$$\text{Cont}^{\mathsf{A}}_\Gamma (e) = 
    (-1)^{k+l}\left[e^{-\frac{\mu \lambda_i}{x}-\frac{\mu \lambda_j}{y}}e_i\left(\overline{\mathds{V}}_{ij}-\frac{\delta_{ij}}{e_i(x+y)}\right)e_j
    \right]_{x^{k-1} y^{l-1}}\, .$$
 Using also the equation
 \begin{align*}
     e_i \overline{\mathds{V}}_{ij} (x, y) e_j  = 
\frac{\sum _{r=0}^2 \overline{\mathds{S}}_i (\phi_r)|_{z=x} \, \overline{\mathds{S}}_j (\phi ^r )|_{z=y}}{x+ y}\, ,
 \end{align*}
we write $\text{Cont}^{\mathsf{A}}_\Gamma (e)$
as
 \begin{align*}
\left[(-1)^{k+l} e^{-\frac{\mu \lambda_i}{x}-\frac{\mu \lambda_j}{y}}\sum_{r=0}^2\overline{\mathds{S}}_i(\phi_r)|_{z=x}\, \overline{\mathds{S}}_j(\phi^r)|_{z=y}
\right]_{x^{k}y^{l-1}-x^{k+1}y^{l-2}+
\ldots +(-1)^{k-1} x^{k+l-1}}
\end{align*}
where the subscript signifies a (signed) sum
of the respective coefficients.
If we substitute the asymptotic expansions \eqref{VS} for
$$\overline{\mathds{S}}_i(1)\, , \ \
\overline{\mathds{S}}_i(H)\, , \ \
\overline{\mathds{S}}_i(H^2)
$$ in the above expression, the Lemma follows from Proposition \ref{RPoly} and Lemma \ref{RPoly2}.
\end{proof}

\subsection{Legs}
Using the contribution formula of Proposition \ref{VEL},
\begin{eqnarray*}
     \text{Cont}^{\mathsf{A}}_\Gamma(l)
    &=&
    (-1)^{\mathsf{A}(l)-1}\left[e^{-\frac{\lann1,1\rann^{\mathsf{p}(l),0+}_{0,2}}{z}}
       \overline{\mathds{S}}_{\mathsf{p}(l)}(H^{k_l})\right]_{z^{\mathsf{A}(l)-1}}\, ,
 \end{eqnarray*}
 we easily conclude
 
\begin{enumerate}
 \item[$\bullet$]
 when the insertion at the marking $l$ is $H^0$,$$\text{Cont}^{\mathsf{A}}_\Gamma(l)\in
\CC(\lambda_0,\lambda_1,\lambda_2)[L^{\pm1}]\, ,$$
 \item[$\bullet$]
 when the insertion at the marking $l$ is $H^1$,
 $$C_1 \cdot \text{Cont}^{\mathsf{A}}_\Gamma(l)\in
\CC(\lambda_0,\lambda_1,\lambda_2)[L^{\pm1}]\, ,$$
\item[$\bullet$]
 when the insertion at the marking $l$ is $H^2$,
 $$C_1C_2 \cdot \text{Cont}^{\mathsf{A}}_\Gamma(l)\in
\CC(\lambda_0,\lambda_1,\lambda_2)[L^{\pm1},X]\, .$$
\end{enumerate}



\section{Holomorphic anomaly for $K\PP^2$}
\label{hafp}

\subsection{Proof of Theorem \ref{ooo}}

By definition, we have
\begin{equation}\label{ffww}
A_2(q)= \frac{1}{L^3}\left(3X
+1 -\frac{L^3}{2}\right)\, .
\end{equation}
Hence, statement (i),
$$\mathcal{F}_g^{\mathsf{SQ}} (q) \in \mathbb{C}[L^{\pm1}][A_2]\, ,$$
follows from Proposition \ref{VE}
and  Lemmas \ref{L1} - \ref{L2}.
Statement (ii),
$\mathcal{F}_g^{\mathsf{SQ}}$ has at most degree $3g-3$ with respect to $A_2$, holds since a stable graph of genus $g$ has at most $3g-3$ edges.
Since 
$$\frac{\partial}{\partial T} = \frac{q}{C_1}\frac{ \partial}{\partial q}\,, $$
statement (iii),
\begin{equation}\label{vvtt}
\frac{\partial^k \mathcal{F}_g^{\mathsf{SQ}}}{\partial T^k}(q) \in \mathbb{C}[L^{\pm1}][A_2][C_1^{-1}]\, ,
\end{equation}
follows since the ring
$$\mathbb{C}[L^{\pm1}][A_2]=\mathbb{C}[L^{\pm1}][X]$$
is closed under the action of the differential operator $$\DD=q\frac{\partial}{\partial q}\, $$
by \eqref{drule}.
The degree of $C_1^{-1}$ in \eqref{vvtt}
is $1$ which yields statement (iv).
\qed

\subsection{Proof of Theorem \ref{ttt}}
\label{prttt}

Let $\Gamma \in \mathsf{G}_{g}(\PP^2)$ be a decorated graph. Let us fix an edge $f\in\mathsf{E}(\Gamma)$:
\begin{enumerate}
\item[$\bullet$] if $\Gamma$ is connected after 
deleting $f$, denote the resulting graph by $$\Gamma^0_f\in \mathsf{G}_{g-1,2}(\PP^2)\, ,$$
\item[$\bullet \bullet$] if $\Gamma$ is disconnected after deleting $f$, denote the resulting two graphs by $$\Gamma^1_f\in \mathsf{G}_{g_1,1}(\PP^2) \ \ \ 
\text{and}\ \ \  \Gamma^2_f\in \mathsf{G}_{g_2,1}(\PP^2)$$
where $g=g_1+g_2$.
\end{enumerate}
There is no canonical
order for the 2 new markings. 
We will always sum over the 2 labellings. So more precisely, the graph
$\Gamma^0_f$ in case $\bullet$
should be viewed as sum
of 2 graphs
$$\Gamma^0_{f,(1,2)} +
\Gamma^0_{f,(2,1)}\, .$$
Similarly, in case $\bullet\bullet$,
we will sum over the ordering of $g_1$ and $g_2$. As usually, the summation
will be later compensated by a factor of
$\frac{1}{2}$ in the formulas.

By Proposition \ref{VE}, we have
the following formula for the contribution 
of the graph $\Gamma$ to the stable quotient
theory of $K\PP^2$,
 $$
 \text{Cont}_\Gamma
     =\frac{1}{|\text{Aut}(\Gamma)|}
     \sum_{\mathsf{A} \in \ZZ_{\ge 0}^{\mathsf{F}}} \prod_{v\in \mathsf{V}} 
     \text{Cont}^{\mathsf{A}}_\Gamma (v)
     \prod_{e\in \mathsf{E}} \text{Cont}^{\mathsf{A}}_\Gamma(e)\, .
 $$


Let $f$ connect the $\T$-fixed points $p_i, p_j \in \PP^2$. Let
the $\mathsf{A}$-values of the respective
half-edges be $(k,l)$. By Lemma \ref{L2}, we have
\begin{equation}\label{Coeff}
\frac{\partial \text{Cont}^{\mathsf{A}}_\Gamma(f)}{\partial X} = (-1)^{k+l}\frac{3 R_{1\,k-1} R_{1\,l-1}}{L\lambda_i^{k-2}\lambda_j^{l-2}}\, .
\end{equation}

\noindent $\bullet$ If $\Gamma$ is connected after deleting $f$, we have
\begin{multline*}
\frac{1}{|\text{Aut}(\Gamma)|}
     \sum_{\mathsf{A} \in \ZZ_{\ge 0}^{\mathsf{F}}} 
     \left(\frac{L^3}{3C^2_1}\right)
     \frac{\partial {\text{Cont}}^{\mathsf{A}}_\Gamma(f)}{\partial X} 
     \prod_{v\in \mathsf{V}} 
     \text{Cont}^{\mathsf{A}}_\Gamma (v)
     \prod_{e\in \mathsf{E},\, e\neq f} \text{Cont}^{\mathsf{A}}_\Gamma(e) \\=
\frac{1}{2} \,
\text{Cont}_{\Gamma^0_f}(H,H) \, .
\end{multline*}
The derivation is simply by using \eqref{Coeff} on the left
and Proposition \ref{VEL} on the right.

\vspace{5pt}
\noindent $\bullet\bullet$
If $\Gamma$ is disconnected after deleting $f$, we obtain
\begin{multline*}
\frac{1}{|\text{Aut}(\Gamma)|}
     \sum_{\mathsf{A} \in \ZZ_{\ge 0}^{\mathsf{F}}} 
     \left(\frac{L^3}{3C^2_1}\right)
     \frac{\partial {\text{Cont}}^{\mathsf{A}}_\Gamma(f)}{\partial X} 
     \prod_{v\in \mathsf{V}} 
     \text{Cont}^{\mathsf{A}}_\Gamma (v)
     \prod_{e\in \mathsf{E},\, e\neq f} \text{Cont}^{\mathsf{A}}_\Gamma(e)\\
=\frac{1}{2}\,
\text{Cont}_{\Gamma^1_f}(H) \,
\text{Cont}_{\Gamma^2_f}(H)\, 
\end{multline*}
by the same method. 

By combining the above two equations for all 
the edges of all the graphs $\Gamma\in \mathsf{G}_g(\PP^2,d)$
and using the vanishing
\begin{align*}
\frac{\partial {\text{Cont}}^{\mathsf{A}}_\Gamma(v)}{\partial X}=0
\end{align*}
of Lemma \ref{L1}, we obtain
\begin{multline}\label{greww}
\left(\frac{L^3} {3C^2_1}\right) \frac{\partial}{\partial X} 
 \lan  \ran^{\mathsf{SQ}}_{g,0}= \frac{1}{2}\sum_{i=1}^{g-1} \lan H\ran^{\mathsf{SQ}}_{g-i,1}
\lan H \ran^{\mathsf{SQ}}_{i,1} + \frac{1}{2} \lan H,H\ran^{\mathsf{SQ}}_{g-1,2}\, .
\end{multline}
We have followed here the notation of Section \ref{holp2}.
The equality \eqref{greww} holds in the ring $\mathbb{C}[L^{\pm}][A_2,C_1^{-1}]$.

Since $A_2=\frac{1}{L^3}(3X+1-\frac{L^3}{2})$ and 
$\lan \, \ran^{\mathsf{SQ}}_{g,0}=\mathcal{F}_g^{\mathsf{SQ}}$,
the left side of \eqref{greww} is, by the chain rule,
$$\frac{1}{C_1^2} \frac{\partial \mathcal{F}_g^{\mathsf{SQ}}}{\partial A_2} \in \mathbb{C}[L^{\pm}][A_2,C_1^{-1}]\, .$$ 
On the right side of \eqref{greww}, we have 
$$ \lan H  \ran^{\mathsf{SQ}}_{g-i,1}\, =\, \mathcal{F}_{g-i,1}^{\mathsf{SQ}}(q)\, =\,  \mathcal{F}^{\mathsf{GW}}_{g-i,1}(Q(q))\, ,$$
where the first equality is by definition and the second is by
wall-crossing \eqref{3456}. Then,
$$\mathcal{F}^{\mathsf{GW}}_{g-i,1}(Q(q))\ = \ \frac{\partial\mathcal{F}^{\mathsf{GW}}_{g-i}}{\partial T}(Q(q)) \ =\ 
\frac{\partial\mathcal{F}^{\mathsf{SQ}}_{g-i}}{\partial T}(q) 
$$
where the first equality is by the divisor equation in
Gromov-Witten theory and the second is again by wall-crossing
\eqref{3456}, so we conclude
$$ \lan H  \ran^{\mathsf{SQ}}_{g-i,1} =\frac{\partial\mathcal{F}^{\mathsf{SQ}}_{g-i}}{\partial T}(q)\, \in \mathbb{C}[[q]]\, .$$
Similarly, we obtain
\begin{eqnarray*}
 \lan H  \ran^{\mathsf{SQ}}_{i,1} &=&\frac{\partial\mathcal{F}^{\mathsf{SQ}}_{i}}{\partial T}(q)\, 
\, \in \mathbb{C}[[q]]\, ,
\\
 \lan H,H  \ran^{\mathsf{SQ}}_{g-1,2} &=&\frac{\partial^2\mathcal{F}^{\mathsf{SQ}}_{g-1}}{\partial T^2}(q)\,
\, \in \mathbb{C}[[q]]\, .
\end{eqnarray*}
Together, the above equations transform \eqref{greww} into 
exactly the holomorphic anomaly equation of Theorem \ref{ttt},
$$\frac{1}{C_1^2}\frac{\partial \mathcal{F}_g^{\mathsf{SQ}}}{\partial{A_2}}(q)
= \frac{1}{2}\sum_{i=1}^{g-1} 
\frac{\partial \mathcal{F}_{g-i}^{\mathsf{SQ}}}{\partial{T}}(q)
\frac{\partial \mathcal{F}_i^{\mathsf{SQ}}}{\partial{T}}(q)
+
\frac{1}{2}
\frac{\partial^2 \mathcal{F}_{g-1}^{\mathsf{SQ}}}{\partial{T}^2}(q)\,
$$
as an equality in $\mathbb{C}[[q]]$.

The series $L$ and $A_2$ are expected to be algebraically independent. Since
we do not have a proof
of the independence, to lift holomorphic anomaly equation to the equality
$$\frac{1}{C_1^2}\frac{\partial \mathcal{F}_g^{\mathsf{SQ}}}{\partial{A_2}}
= \frac{1}{2}\sum_{i=1}^{g-1} 
\frac{\partial \mathcal{F}_{g-i}^{\mathsf{SQ}}}{\partial{T}}
\frac{\partial \mathcal{F}_i^{\mathsf{SQ}}}{\partial{T}}
+
\frac{1}{2}
\frac{\partial^2 \mathcal{F}_{g-1}^{\mathsf{SQ}}}{\partial{T}^2}\,
$$
in the ring $\mathbb{C}[L^{\pm1}][A_2,C_1^{-1}]$, we must
prove
the equalities 
\begin{equation}\label{pp33p}
 \lan H  \ran^{\mathsf{SQ}}_{g-i,1} =\frac{\partial\mathcal{F}^{\mathsf{SQ}}_{g-i}}{\partial T}\,,  \ \ \ \ 
 \lan H  \ran^{\mathsf{SQ}}_{i,1} = \frac{\partial\mathcal{F}^{\mathsf{SQ}}_{i}}{\partial T}\, ,
\end{equation}
$$ \lan H,H  \ran^{\mathsf{SQ}}_{g-1,2}= \frac{\partial^2\mathcal{F}^{\mathsf{SQ}}_{g-1}}{\partial T^2}\,
$$
in the ring $\mathbb{C}[L^{\pm}][A_2,C_1^{-1}]$.
The lifting will be proven in Section \ref{llfftt}
below.

We do not study the genus 1 unpointed series $\mathcal{F}^{\mathsf{SQ}}_1(q)$ in the paper, so we take
\begin{eqnarray*}
 \lan H  \ran^{\mathsf{SQ}}_{1,1} &=&\frac{\partial\mathcal{F}^{\mathsf{SQ}}_{1}}{\partial T}\, ,\\
 \lan H,H  \ran^{\mathsf{SQ}}_{1,2} &=&\frac{\partial^2\mathcal{F}^{\mathsf{SQ}}_{1}}{\partial T^2}\, .
\end{eqnarray*}
as definitions of the right side in the genus 1 case.
There is no difficulty in calculating these series explicitly
using Proposition \ref{VEL},

\begin{eqnarray*}
\frac{\partial \mathcal{F}_1^{\mathsf{SQ}}}{\partial{T}} &= &-\frac{1}{6 C_1}L^3 A_2\, ,\\
\frac{\partial^2 \mathcal{F}_1^{\mathsf{SQ}}}{\partial{T}^2}& = &
\frac{1}{C_1} \DD\left(-\frac{1}{6 C_1}L^3 A_2\right)\,.
\end{eqnarray*}

\subsection{Lifting} \label{llfftt}
We write the first two equalities
in \eqref{pp33p} 
together as
\begin{equation}\label{fss19}
\lan H  \ran^{\mathsf{SQ}}_{h,1} =\frac{\partial\mathcal{F}^{\mathsf{SQ}}_{h}}{\partial T}\, .
\end{equation}
The formula
of Proposition \ref{VEL} for
the left side of \eqref{fss19}
is a
summation over graphs $\Gamma\in \mathsf{G}_{h,1}(\proj^2)$.
Stabilization yields canonical map,
$$\mathsf{G}_{h,1}(\proj^2) \rightarrow \mathsf{G}_{h,0}(\proj^2)\, , \ \ \ \ \Gamma \mapsto
\widetilde{\Gamma}\, ,$$
obtained by forgetting the marking $1$.
\begin{enumerate}
\item[$\bullet$] If the marking $1$ is carried by a vertex $v$
which is stable without the marking, $\widetilde{\Gamma}$
is simply obtained by removing the marking.
The marking falls to the corresponding
vertex $v$ of $\widetilde{\Gamma}$.
\item[$\bullet$] 
If the marking $1$ is carried by a vertex $v$
which is unstable without the marking, then
$v$ is contracted in the stabilization.
The marking $1$
falls to a 
 unique edge of $\widetilde{\Gamma}$.
\end{enumerate} 
For a fixed edge $\widetilde{f}$ of $\widetilde{\Gamma}$ there are exactly 3
graphs of $\mathsf{G}_{h,1}(\proj)$ for 
which the marking $1$ falls to 
$\widetilde{f}$. These come from the
3 possible $\mathsf{p}$ values of the
contracted vertex.

If we start with an edge $\widetilde{f}$ of
$\widetilde{\Gamma}\in\mathsf{G}_{h,0}(\proj^2)$
connecting{\footnote{The analysis of the self edge case is identical
and left to the reader.}} two vertices $v$ and $v'$, there are 5 associated graphs in $\mathsf{G}_{h,1}(\proj^2)$:
\begin{enumerate}
    \item [$\bullet$]
let $\Gamma_{v}^{\widetilde{f}}, \Gamma_{v'}^{\widetilde{f}}\in \mathsf{G}_{h,1}(\proj^2)$
be
the graphs where the marking $1$ falls to the 
respective vertices of $\widetilde{\Gamma}$,
\item[$\bullet$] let $\Gamma_{1}^{\widetilde{f}},
\Gamma_{2}^{\widetilde{f}},
\Gamma_{3}^{\widetilde{f}}\in \mathsf{G}_{h,1}(\proj^2)$ be 
the three graphs where the marking $1$
falls to the edge $\tilde{f}$ of $\widetilde{\Gamma}$.
\end{enumerate}


The right side of \eqref{fss19}
may also be written as a summation
over graphs $\Gamma\in \mathsf{G}_{h,1}(\proj^2)$.
The formula of Proposition \ref{VEL} for
$\mathcal{F}^{\mathsf{SQ}}_{h}$ is a
summation over graphs  
$$\widetilde{\Gamma}\in \mathsf{G}_{h,0}(\proj^2)\, .$$
We then view the action of the
derivative $\frac{\partial}{\partial T}$ 
as producing the marking $1$. 
If the derivative acts on a vertex contribution of $\widetilde{\Gamma}$, the marking is distributed to
that
vertex. If the derivative acts on an edge $\widetilde{f}$
contribution, we view the differentiation 
as accounting for the sum of the 3 graphs 
\begin{equation} \label{tripp}
\Gamma_{1}^{\widetilde{f}},
\Gamma_{2}^{\widetilde{f}},
\Gamma_{3}^{\widetilde{f}}\in
\mathsf{G}_{h,1}(\proj^2)\, .
\end{equation}

 

We now apply Proposition \ref{VEL} via the  above
analysis to the difference
\begin{equation}\label{ff345f}
    \frac{\partial \mathcal{F}^{\text{SQ}}_h}{\partial T}-\lan H\ran^{\text{SQ}}_{h,1}\,.
\end{equation}
The result is a summation of
 contributions corresponding to graphs $\Gamma\in\mathsf{G}_{h,1}(\PP^2)$
 where all triples \eqref{tripp} are considered
 contributing together.



%

For graphs $\Gamma\in \mathsf{G}_{h,1}(\proj^2)$
for which the marking $1$ falls to a vertex $v$ in
$\widetilde{\Gamma}$, we will distribute the
vertex contribution 
naturally to
the incident edges of $v \in \widetilde{\Gamma}$
by the following method.
Proposition \ref{q2q2} and the string equation
yield a local version of divisor equation in the ring $\CC[L^{\pm 1}][A_2,C_1^{-1}]$. For $\gamma\in H^*(\overline{M}_{g,n})$,
\begin{multline}\label{DE}
    \frac{\partial}{\partial T}\ppl
\psi_1^{a_1}, \ldots,
\psi_n^{a_n}
\, \Big|\, \gamma\,
  \ppr_{h,n}^{\mathsf{p}(v),0+} =\\
    \sum_{k\geq0} \ppl
\psi_1^{a_1}, \ldots,
\psi_n^{a_n},\psi_{n+1}^k
\, \Big|\, \gamma\,
  \ppr_{h,n+1}^{\mathsf{p}(v),0+}\frac{(-1)^{k}\lambda_{\mathsf{p}(v)}^{1-k}LR_{1k}}{C_1}\,\\
  -\sum_{j=0}^{n} \ppl
\psi_1^{a_1}, \ldots,\psi_1^{a_j-1},\ldots,
\psi_n^{a_n}
\, \Big|\, \gamma\,
  \ppr_{h,n}^{\mathsf{p}(v),0+}\frac{\lambda_{\mathsf{p}(v)} L}{C_1}\,
  ,\end{multline} 
While Proposition \ref{q2q2} was stated only for genus 0 invariants without $\psi$ insertions, the
same result holds for all genera $h$ with $\psi$ by exactly the same argument,

    \begin{multline}\label{VWP}
\ppl
\psi_1^{a_1}, \ldots,
\psi_n^{a_n}
\, \Big|\, \gamma\,
  \ppr_{h,n}^{\mathsf{p}(v),0+} = \\
{\small{ \left(\sum_{k\geq 0}\frac{1}{k!}\int_{\overline{M}_{h,n+k}}\psi_1^{a_1}\cdots\psi_n^{a_n} \cdot\gamma\cdot T(\psi_{n+1})\cdots T(\psi_{n+k})\right)\Big|_{t_0=0,t_1=0,t_{j\ge 2}=(-1)^j\frac{R_{j-1}}{\lambda_i^{j-1}}}\, .}}
 \end{multline}

By Proposition \ref{VEL} with \eqref{DE}, the $\mathsf{A}$-valued contribution{\footnote{For $\sum_{i=1}^n a_i = 3g-3+n+1-r$,  where $r$ is degree of $\gamma\in H^*(\overline{M}_{g,n})$, we use the following equation which can be easily checked by \eqref{VWP} and the string equation,
\begin{equation*} \ppl
\psi_1^{a_1},\ldots,
\psi_n^{a_n},1
\, \Big|\, \gamma\,
  \ppr_{\mathsf{g}(v),n+1}^{\mathsf{p}(v),0+} =
  \sum_{j=1}^n\ppl
\psi_1^{a_1}, \ldots,\psi_j^{a_j-1},\ldots,
\psi_n^{a_n}
\, \Big|\, \gamma\,
  \ppr_{\mathsf{g}(v),n}^{\mathsf{p}(v),0+} \,.\end{equation*}}} 
to
\begin{align*}
    \frac{\partial \mathcal{F}^{\text{SQ}}_h}{\partial T}-\lan H\ran^{\text{SQ}}_{h,1}\,,
\end{align*}
of the vertex $v$ is
$$-\sum_{j=1}^n \ppl
\psi_1^{a_1}, \ldots,\psi_j^{a_j-1},\ldots,
\psi_n^{a_n}
\, \Big|\, \mathsf{H}_{\mathsf{g}(v)}^{\mathsf{p}(v)}\,
  \ppr_{\mathsf{g}(v),n}^{\mathsf{p}(v),0+} \frac{\lambda_{\mathsf{p}(v)}L}{C_1}\, \,.$$
Here, we have used also the following equation for the $H$ insertion
 on the leg $l$ at the vertex $v$:
\begin{align*}
     \text{Cont}^{\mathsf{A}}_\Gamma(l)&
    =
    (-1)^{\mathsf{A}(l)-1}\left[e^{-\frac{\lann1,1\rann^{\mathsf{p}(l),0+}_{0,2}}{z}}
       \overline{\mathds{S}}_{\mathsf{p}(l)}(H)\right]_{z^{\mathsf{A}(l)-1}}\, \\
    &=(-1)^{\mathsf{A}(l)-1}\frac{\lambda^{-\mathsf{A}(l)}_{\mathsf{p}(v)}L R_{1\,\mathsf{A}(l)-1}}{C_1}\, .
 \end{align*}  
The location of $\psi_j$ with exponent $a_j-1$ exactly tell us to which edge we distribute.

After applying Proposition \ref{VEL} to \begin{align*}
    \frac{\partial \mathcal{F}^{\text{SQ}}_h}{\partial T}-\lan H\ran^{\text{SQ}}_{h,1}\,,
\end{align*}
with above vertex term distribution, the
sum of the contributions of the 5 graphs
$$\Gamma^{\widetilde{f}}_{v}\, ,\ \Gamma^{\widetilde{f}}_{v'}\,,\ \Gamma^{\widetilde{f}}_1\,,\
\Gamma^{\widetilde{f}}_2\, ,\ \Gamma^{\widetilde{f}}_3$$
related to the edge $\widetilde{f}$ of $\widetilde{\Gamma}$ can be written as:
\begin{align}\label{EE}
    &-\frac{1}{L}\DD \left[e^{-\frac{\mu \lambda_i}{x}-\frac{\mu \lambda_j}{y}}e_i\overline{\mathds{V}}_{ij}(x,y)e_j\right]_{x^k y^l} +\\ \nonumber&\sum_{\alpha=0}^{2}\left[e^{-\frac{\mu \lambda_i}{x}-\frac{\mu \lambda_\alpha}{z}}e_i\overline{\mathds{V}}_{i\alpha}(x,z)e_\alpha\right]_{x^{k}}
    \left[e^{-\frac{\mu \lambda_\alpha}{z}-\frac{\mu \lambda_j}{y}}e_\alpha\overline{\mathds{V}}_{\alpha j}(z,y)e_j\right]_{y^{l}}\\
    \nonumber&- \left[e^{-\frac{\mu \lambda_i}{x}-\frac{\mu \lambda_j}{y}}e_i\overline{\mathds{V}}_{ij}(x,y)e_j\right]_{x^{k+1}y^{l}}-\left[e^{-\frac{\mu \lambda_i}{x}-\frac{\mu \lambda_j}{y}}e_i\overline{\mathds{V}}_{ij}(x,y)e_j\right]_{x^k y^{l+1}} .
\end{align}
The first two terms come from sum of the triple
$\Gamma^{\widetilde{f}}_1$,
$\Gamma^{\widetilde{f}}_2$, 
$\Gamma^{\widetilde{f}}_3$,
and last two terms come from $\Gamma^{\widetilde{f}}_{v}$, $\Gamma^{\widetilde{f}}_{v'}$.
The vanishing of above sum in the ring $\CC[L^{\pm 1}][A_2,C_1^{-1}]$ is easily obtained
using Lemma \ref{R} (including relation \eqref{drule}). 

Since equation \eqref{DE} and the vanishing of \eqref{EE} holds in the ring $\CC[L^{\pm 1}][A_2,C_1^{-1}]$, we have proven the identity
$$  \frac{\partial \mathcal{F}^{\text{SQ}}_h}{\partial T}-\lan H\ran^{\text{SQ}}_{h,1}=0$$
in the ring $\CC[L^{\pm 1}][A_2,C_1^{-1}]$.
The proof of
$$ \frac{\partial^2\mathcal{F}^{\mathsf{SQ}}_{h}}{\partial T^2}-\lan H,H  \ran^{\mathsf{SQ}}_{h,2}=0\,$$
in the ring $\CC[L^{\pm 1}][A_2,C_1^{-1}]$
is identical.
The proof of the lifting completes the proof
of Theorem \ref{ttt}. \qed

\subsection{Explicit calculations in genus 2}
We present here the full calculation of $\mathcal{F}^{SQ}_2$ for 
$K\PP^2$. 
The $7$ stable graphs of genus 2 are:

\includegraphics[scale=0.47]{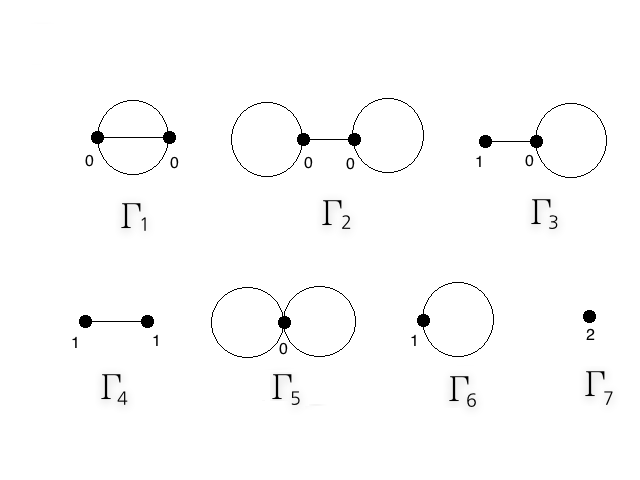}

\noindent The full contribution
of each stable graph $\Gamma_i$  is obtained by summing the contributions  of all
possible decorations:

\begin{small}
\begin{align*}
    \text{Cont}_{\Gamma_1} =&  \frac{24-12L+6L^2-61L^3+12L^4-3L^5+54L^6-3L^7-17L^9}{2592L^3}\\
    &+
    \frac{12-4L+L^2-20L^3+2L^4+9L^6}{144L^3}X\\
    &+\frac{6-L-5L^3}{24L^3}X^2\\
    &+\frac{1}{4L^3}X^3\, ,\\
    \text{Cont}_{\Gamma_2} =& \frac{24-28L+10L^2-45L^3+36L^4-7L^5+26L^6-11L^7-5L^9}{1728L^3}\\
    &+
    \frac{36-28L+5L^2-44L^3+18L^4+13L^6}{288L^3}X\\
    &+\frac{18-7L-11L^3}{48L^3}X^2\\
    &+\frac{3}{8L^3}X^3\, ,\\ \\
    \end{align*}
\begin{align*}
    \text{Cont}_{\Gamma_3} = &\frac{288-190L-25L^2-364L^3+145L^4+74L^5+97L^6-25L^8}{20736L^2}\\
    &+\frac{288-95L-24L^2-194L^3+25L^5}{3456L^2}X\\
    &+\frac{1}{8L^2}X^2\, , \\ \\
    \text{Cont}_{\Gamma_4} =& \frac{2592-541L-864L^2-2229L^3+720L^4+897L^5-575L^7}{746496L}\\
    &+\frac{1}{96L}X\, , \\ \\
    \text{Cont}_{\Gamma_5} = &\frac{12-8L-11L^2-8L^3+5L^4+16L^5-L^6-5L^8}{1728L^2}\\
    &+\frac{3-L-3L^2-L^3+2L^5}{72L^2}X\\
    &+\frac{1-L^2}{16L^2}X^2\,, \\ \\
    \text{Cont}_{\Gamma_6} = &\frac{138+143L-204L^2-135L^3-222L^4+201L^5+79L^7}{62208L}\\
    &+\frac{23+24L-22L^2-25L^4}{3456L}X\, , \\ \\
    \text{Cont}_{\Gamma_7} = &\frac{281+4320L+1785L^2-2736L^3-3765L^4+2059L^6}{3732480} \, .
\end{align*}
\end{small}

\noindent After summing the above contributions, we obtain the following
result
which exactly matches \cite[(4.35)]{ASYZ}.

\begin{Prop} \label{gen222} The stable quotient series for $K\PP^2$ in genus 2 is 
 \begin{align*}
     \mathcal{F}^{\mathsf{SQ}}_{2}=&\frac{400-959L^3+784L^6-216L^9}{17280L^3}+\left(-\frac{1}{3}+\frac{5}{24L^3}+\frac{13L^3}{96}\right)X\\
     &+\left(-\frac{1}{2}+\frac{5}{8L^3}\right)X^2+\frac{5}{8L^3}X^3.
 \end{align*}
\end{Prop}

 The B-model potential function for genus 2 is also calculated in \cite[Section 3.5]{KZ}
 according to graph contributions obtained from B-model physics. Though the calculation exactly matches our result in total, the individual graph contributions do not match. 
  The relationship between the different graph contributions would be good to
  understand.

\subsection{Bounding the degree}

 The degrees in $L$ of the terms of
 $$\mathcal{F}^{\mathsf{SQ}}_{g}\in \mathbb{C}[L^{\pm1}][A_2]$$ for $K\proj^2$ 
 always fall in the range
 \begin{equation}\label{ds33}
 [9-9g, 6g-6]\, .
 \end{equation}
In particular, the constant (in $A_2$) term of $\mathcal{F}^{\mathsf{SQ}}_{g}$
missed by the
holomorphic anomaly equation for $K\proj^2$ is a
Laurent polynomial in $L$ with degrees in the range
\eqref{ds33}.
The bound \eqref{ds33} is a consequence of Proposition \ref{VE}, 
 the vertex and edge analysis of Section \ref{svel}, and the
 following result.
 
\begin{Lemma}
 The degrees in $L$ of $R_{ip}$ fall in the range
 \begin{align*}
     [-i,2p]\, .
 \end{align*}
\end{Lemma} 

\begin{proof}
 The proof for the functions $R_{0p}$ follows from the arguments of \cite{ZaZi}. The proof for $R_{1p}$ and $R_{2p}$ follows  from Lemma \ref{R}.
\end{proof}

For $\mathcal{F}^{\mathsf{SQ}}_{2}$,  the $L$ degrees can be seen to vary between
 $0$ and $6$ in the formula of Propositon \ref{gen222} when
 rewritten in terms of $A_2$ using \eqref{ffww}.
 The sharper range
\begin{equation*}
 [0, 6g-6]\, 
 \end{equation*}
 proposed in \cite{ASYZ} 
for the $L$ degrees of $\mathcal{F}^{\mathsf{SQ}}_{g}$
is found in examples. How to derive the sharper bound
from properties of the functions $R_{ip}$ is
an interesting question.

\section{Holomorphic anomaly for $K\mathbb{P}^2$ with insertions}
\label{hakp2}
\subsection{Insertions}
Since the stable quotient theory
of $K\mathbb{P}^2$ has virtual dimension 0, insertions
do not play a significant role in the nonequivariant
theory. However, for the
{\em equivariant} stable quotient theory of
$K\mathbb{P}^2$, insertions of all dimensions can be studied.
After the specialization of torus weights
\begin{equation}\label{kk123}
\lambda_0=1\, , \ \ \lambda_1=\zeta\, , \ \ \lambda_2=\zeta^2\, , 
\end{equation}
in the equivariant theory (with $\zeta^3=1$) , 
we obtain a {\em numerical} theory of $K\mathbb{P}^2$ with
arbitrary insertions. Define the series
\begin{eqnarray*}
\mathcal{F}_{g,n}^{\mathsf{SQ}}[a,b,c] & = &
\lan  \,\tau_0(1)^a \tau_0(H)^b\tau_0(H^2)^c\, \ran^{\mathsf{SQ}}_{g,n}\\ & = & \nonumber
\sum_{d=0}^\infty  q^d
\int_{[\overline{Q}_{g,n}(K\PP^2,d)]^{vir}} \prod^{a+b}_{i=a+1}\text{ev}_{i}^*(H) \prod_{i=a+b+1}^{a+b+c}\text{ev}_{i}^*(H^2)\, , 
\end{eqnarray*}
with $n=a+b+c$.

Our proof of Theorem \ref{ooo} immediately yields the parallel
results for the stable quotient series with insertions:

\begin{enumerate}
\item[(i)]
$\mathcal{F}^{\mathsf{SQ}}_{g,n}[a,b,c]\in \mathbb{C}[L^{\pm1}][A_2,C_1,C_1^{-1}]$
for  $2g-2+n > 0$, \vspace{5pt}
\item[(ii)]
$\mathcal{F}^{\mathsf{SQ}}_{g,n}[a,b,c]$
 is of degree $\leq 3g-3+c$  in $A_2$,
 \vspace{5pt}
\item[(iii)]
$\frac{\partial^k \mathcal{F}_{g,n}^{\mathsf{SQ}}[a,b,c]}{\partial T^k} \in \mathbb{C}[L^{\pm1}][A_2,C_1,C_1^{-1}]$ for  $2g-2+n\geq 0$ and  $k\geq 1$,
\vspace{5pt}

\item[(iv)]
${\frac{\partial^k \mathcal{F}_{g,n}^{\mathsf{SQ}}[a,b,c]}{\partial T^k}}$ is homogeneous of degree $k+b-c$
in  $C_1^{-1}$.
\end{enumerate}
For example, a computation by Proposition  \ref{VEL} yields
$$\mathcal{F}_{0,3}^{\mathsf{SQ}}[0,0,3]=
\lan \,\tau_0(H^2)^3\, \ran^{\mathsf{SQ}}_{g,n} = -\frac{1}{3}\left(\frac{C_1}{L}\right)^3 \, .$$
A natural question 
is whether a holomorphic anomaly
equation of the form of Theorem \ref{ttt} holds for 
$\mathcal{F}_{g,n}^{\mathsf{SQ}}[a,b,c]$. The answer is {\em yes}, but
with an additional descendent term.

Insertions of higher powers of $H$ can also be included in the stable
quotient theory of $K\mathbb{P}^2$.
However, because of the specialization of torus weights
\eqref{kk123}, insertions of $H^k$ for $k\geq 3$ can be
reduced to insertions of $1,H,H^2$ via the relation
$$\tau_0(H^r) =\tau_0(H^s) \ \ \ \text{for} \ \ \ r\equiv s \mod 3\, .$$

\subsection{Holomorphic anomaly equation}
Let $\pi$ be a morphism to the moduli space of stable curves
determined by the domain,
$$\pi:\overline{Q}_{g,n}(K\PP^2,d)\rightarrow \overline{M}_{g,n}
\, .$$
Define following the series of stable quotient invariants with
descendents,
\begin{equation*}
\mathcal{F}_{g,n}^{\mathsf{SQ}}[a,b,c,\delta]  = 
\lan \,\tau_0(1)^a \tau_0(H)^b\tau_0(H^2)^c
\, \widetilde{\tau}_1(H)^{\delta}\,  \ran^{\mathsf{SQ}}_{g,n=a+b+c+\delta}\, .
\end{equation*}
The descendent $\widetilde{\tau}_1(H)$ here{\footnote{The tilde
is used to indicate the pulled-back cotangent line.}}
corresponds to the insertion
$$\pi^*(\psi_{i})\cdot \text{ev}_{i}^*(H)$$
with respect to the cotangent line {\em pulled-back}
via $\pi$.


\begin{Thm} \label{ss56}
For $2g-2+n > 0$ and a partition $n=a+b+c$, 

 \begin{eqnarray*}
 \frac{1}{C_1^2}\frac{\partial\mathcal{F}^{\mathsf{SQ}}_{g,n}[a,b,c]}{\partial A_2}&=&\ \
 \frac{1}{2}\sum
 \frac{\partial \mathcal{F}^{\mathsf{SQ}}_{g_1,n_1}[a_1,b_1,c_1]}{\partial T} \, \frac{\partial \mathcal{F}^{\mathsf{SQ}}_{g_2,n_2}[a_2,b_2,c_2]}{\partial T}\\
& & +\frac{1}{2}\frac{\partial^2 \mathcal{F}^{\mathsf{SQ}}_{g-1,n}[a,b,c]}{\partial T^2}\\
& & -\frac{1}{3}c\, \mathcal{F}^{\mathsf{SQ}}_{g,n}[a,b,c-1,1]\, .\\
 \end{eqnarray*}
\end{Thm}

The sum in the first term on the right  is over all
genus decompositions
$$g_1+g_2=g$$
and all distributions of the $n$ markings to the
two parts. The point distributions determines
decompositions
$$ n_1+n_2=n\,, \ \ 
a_1+a_2=a\,, \ \ b_1+b_2=b\,, \ \ c_1+c_2=c\,.$$
In fact, each such decomposition occurs
$$\binom{a}{a_1,a_2}\binom{b}{b_1,b_2} \binom{c}{c_1,c_2}$$
times in the sum.
The distributions are required to satisfy 
$$2g_1-2+n_1\geq 0\, , \ \ \ 2g_2-2+n_2\geq 0\, .$$
The unstable genus 0 cases in the sum
are defined by
\begin{eqnarray*}
\frac{\partial\mathcal{F}_{0,2}^{\mathsf{SQ}}(H^r,H^s)}{\partial T}&=& \lan \, H^r,H^s,H\,  \ran_{0,3}^{\mathsf{SQ}}\, ,
\end{eqnarray*}
and the unstable genus 1 cases are defined
as in Section \ref{prttt}.

In the second term on the right, further unstable terms
in genus 0 and 1 may appear. In genus 0, the definitions
are
\begin{eqnarray*}
\frac{\partial^2\mathcal{F}_{0,2}^{\mathsf{SQ}}(H^r,H^s)}{\partial T^2}&=&  \lan\, H^r,H^s,H,H\,  \ran_{0,4}^{\mathsf{SQ}}\, ,\\
\frac{\partial^2\mathcal{F}_{0,1}^{\mathsf{SQ}}(H^r)}{\partial T^2}&=&  \lan\, H^r,H,H\,  \ran_{0,3}^{\mathsf{SQ}}\, .
\end{eqnarray*}
In genus 1, the unstable cases are defined again
as in Section \ref{prttt}.
Together, the first two terms on right 
exactly match the holomorphic anomaly equation of Theorem \ref{ttt} for $K\mathbb{P}^2$ without
insertions.

The inclusion and precise form of the new 
third term was 
motivated by the recent work 
of Oberdieck and Pixton \cite{ObPix} on the
holomorphic anomaly equation for the elliptic curve $E$.
While our theory of $K\mathbb{P}^2$ and the
theory of $E$ appear to have little in common, at least
two features are parallel:
{\em both} targets are Calabi-Yau and
{\em both} theories admit nontrivial insertions{\footnote{The moduli spaces of maps to $E$ have {\em positive} virtual dimension.}}.
Oberdieck and Pixton prove a holomorphic anomaly
equation (at the cycle level) for the elliptic curve where
the differentiation on the left side is with respect to the Eisenstein series
$E_2$ (instead of $A_2$ here for $K\mathbb{P}^2$). The equation
of Oberdieck and Pixton has a third term exactly involving 
a single pulled-back descendent.{\footnote{After attending
Pixton's lecture at the Institute Henri Poincar\'e in
Paris in March 2017, we realized the
parallel term is correct for our $K\mathbb{P}^2$ theory.}}

\subsection{Proof of Theorem \ref{ss56}} The localization formula of Proposition \ref{VEL} can be easily
extended to include the new descendent insertion.

\begin{Prop}\label{VELD} We have
 \begin{multline*}
 \text{\em Cont}_\Gamma(H^{k_1},\ldots,H^{k_n},\pi^*(\psi_{n+1})H)
     =\\\frac{1}{|\text{\em Aut}(\Gamma)|}
     \sum_{\mathsf{A} \in \ZZ_{\ge 0}^{\mathsf{F}}} \prod_{v\in \mathsf{V}} 
     \text{\em Cont}^{\mathsf{A}}_\Gamma (v)
     \prod_{e\in \mathsf{E}} \text{\em Cont}^{\mathsf{A}}_\Gamma(e)
     \prod_{l\in \mathsf{L}} \text{\em Cont}^{\mathsf{A}}_\Gamma(l)\, ,
 \end{multline*}
 where the leg contribution 
 is 
 \begin{eqnarray*}
     \text{\em Cont}^{\mathsf{A}}_\Gamma(l)
    &=&
    (-1)^{\mathsf{A}(l)-1}\left[e^{-\frac{\lann1,1\rann^{\mathsf{p}(l),0+}_{0,2}}{z}}
       \overline{\mathds{S}}_{\mathsf{p}(l)}(H^{k_l})\right]_{z^{\mathsf{A}(l)-1}}\, 
 \end{eqnarray*}
 for $l \in \{1,\ldots,n\}$ and
 
 \begin{eqnarray*}
     \text{\em Cont}^{\mathsf{A}}_\Gamma(l)
    &=&
    (-1)^{\mathsf{A}(l)-1}\left[e^{-\frac{\lann1,1\rann^{\mathsf{p}(l),0+}_{0,2}}{z}}
       z \overline{\mathds{S}}_{\mathsf{p}(l)}(H)\right]_{z^{\mathsf{A}(l)-1}}\,
 \end{eqnarray*}
 for $l=n+1$.
The vertex and edge contributions are same as before.
\end{Prop}

Let $l$ be a leg of $\Gamma$ with insertion 
$H^2$. The $X$ derivative of
the leg contribution of $l$  is
\begin{align*}
    \frac{\partial \text{Cont}^\mathsf{A}_\Gamma(l)}{\partial X}=& \left\{ \begin{array}{rl} (-1)^{a}\frac{L}{C_1C_2}\frac{R_{1\, a-2}}{\lambda_{\mathsf{p}(l)}^{a-3} } & \text{if } a \ge 2  \\
                                    0 & \text{if } a = 1  \, , \end{array}\right.
\end{align*}
where $a=\mathsf{A}(l)$.
Let $\widetilde{l}$ be a leg of $\Gamma$ with insertion 
$\pi^*(\psi)H$. The $X$ derivative of
the leg contribution of $\widetilde{l}$  is 
\begin{align*}                      
    \text{Cont}^\mathsf{A}_\Gamma(\, \widetilde{l}\, )=\left\{ \begin{array}{rl} (-1)^{a+1}\frac{L}{C_1}\frac{R_{1\,a-2}}
    {\lambda_{\mathsf{p}(\, \widetilde{l}\, )}^{a-3}} & \text{if } a \ge 2  \\
                                    0 & \text{if } a = 1  \, , \end{array}\right.
    \end{align*}
    where $a=\mathsf{A}(\, \widetilde{l}\, )$.
Hence, when 
$\mathsf{p}(l)=\mathsf{p}(\, \widetilde{l}\, )$ and
$\mathsf{A}(l)= \mathsf{A}(\, \widetilde{l}\,)$,
we obtain the equation{\footnote{We have used the
identity $L^3=C_1^2C_2$ obtained from \eqref{c1c2l}.}}
\begin{equation*}
 \frac{1}{C_1^2}\frac{\partial \text{Cont}^\mathsf{A}_\Gamma(l)}{\partial A_2} =
    \frac{L^3}{3C_1^2}\frac{\partial \text{Cont}^\mathsf{A}_\Gamma(l)}{\partial X} = -\frac{1}{3}\text{Cont}^\mathsf{A}_\Gamma(\, \widetilde{l}\,)\, ,
\end{equation*}
which explains the third term on the right side of 
holomorphic anomaly equation of Theorem \ref{ss56}. 
 The proof Theorem \ref{ss56} then follows by exactly the same argument used for the proof Theorem \ref{ttt}.  \qed
 \vspace{8pt}

 In fact, the same proof yields a general holomorphic anomaly equation
 for all series including the insertion $\widetilde{\tau}_1(H)$,
 
 \begin{eqnarray*}
 \frac{1}{C_1^2}\frac{\partial\mathcal{F}^{\mathsf{SQ}}_{g,n}[a,b,c,\delta]}{\partial A_2}&=&\ \
 \frac{1}{2}\sum
 \frac{\partial \mathcal{F}^{\mathsf{SQ}}_{g_1,n_1}[a_1,b_1,c_1,\delta_1]}{\partial T} \, \frac{\partial \mathcal{F}^{\mathsf{SQ}}_{g_2,n_2}[a_2,b_2,c_2,\delta_2]}{\partial T}\\
& & +\frac{1}{2}\frac{\partial^2 \mathcal{F}^{\mathsf{SQ}}_{g-1,n}[a,b,c,\delta]}{\partial T^2}\\
& & -\frac{1}{3}c\, \mathcal{F}^{\mathsf{SQ}}_{g,n}[a,b,c-1,\delta+1]\, .\\
 \end{eqnarray*}


\begin{thebibliography}{99}

\bibitem{TopVer}
 M. Aganagic, A. Klemm, M. Mari\~no, C. Vafa, {\em The topological vertex}, Comm. Math. Phys. {\bf 254} (2005), 425--478.


\bibitem{ASYZ} M. Alim, E. Scheidegger, S.-T. Yau, J. Zhou, {\em Special polynomial rings, quasi modular forms and duality of topological strings}, Adv. Theor. Math. Phys. {\bf 18} (2014), 401--467.

\bibitem{BF}
K.~Behrend, B.~Fantechi,
\newblock {\em The intrinsic normal cone,} 
{Invent. Math.} {\bf 128} (1997), 45--88.


\bibitem{BCOV} M. Bershadsky, S. Cecotti, H. Ooguri, and C. Vafa, {\em Holomorphic anomalies in topological field theories}, Nucl. Phys. B{\bf 405}
(1993), 279--304.

\bibitem{BKMP}
V. Bouchard, A. Klemm,
M. Mari\~no, and S. Pasquetti,
{\em Remodelling the B-model},
Comm. Math. Phys. {\bf 287}
(2009), 117--178.


\bibitem{CK} I. Ciocan-Fontanine and B. Kim, {\em Moduli stacks of stable toric quasimaps,}  Adv. in Math. {\bf 225} (2010), 3022--3051.

\bibitem {CKg0}  I. Ciocan-Fontanine and B. Kim,
{\em Wall-crossing in genus zero quasimap theory and mirror maps,}   Algebr. Geom. {\bf 1 } (2014), 400--448.





\bibitem{BigI}   I. Ciocan-Fontanine and B. Kim, {\em Big I-functions} in
 {\em Development of moduli theory Kyoto 2013}, 323--347, Adv. Stud. Pure Math. {\bf 69}, 
 Math. Soc. Japan, 2016. 



\bibitem {CKw}  I. Ciocan-Fontanine and B. Kim, {\em Quasimap wallcrossings
and mirror symmetry}, arXiv:1611.05023.


\bibitem {CKg}  I. Ciocan-Fontanine and B. Kim, 
{\em Higher genus quasimap wall-crossing for semi-positive targets}, JEMS {\bf 19}  (2017),
2051-2102.



\bibitem{CKM} I. Ciocan-Fontanine, B. Kim, and D. Maulik,  {\em Stable quasimaps to GIT quotients,}
  J. Geom. Phys. {\bf 75} (2014), 17--47.

\bibitem{CJR} E. Clader, F. Janda, and Y. Ruan,  {\em Higher genus quasimap wall-crossing
via localization},
  arXiv:1702.03427.

\bibitem{CZ} Y. Cooper and A. Zinger, {\em Mirror symmetry for stable quotients invariants}, Michigan Math. J. {\bf 63} (2014), 571--621.

\bibitem{CosLi} K. Costello and S. Li, {\em Quantum BCOV theory on Calabi-Yau manifolds and the higher genus B-model}, arXiv:1201.4501.

\bibitem{CKatz} D. Cox and S. Katz,
{\em Mirror symmetry and algebraic geometry}, Mathematical Surveys and Monographs {\bf 68}: Amer. Math. Soc., Providence, RI, 1999.

\bibitem{EMO} B.~Eynard,
M.~Mari\~no, and N.~Orantin,
{\em Holomorphic
anomaly and matrix
models}, JHEP {\bf 58} (2007).

\bibitem{FLZ}
B. Fang, M. C.-C. Liu,
and Z. Zong, {\em On the
remodelling conjecture
for toric Calabi-Yau 
3-orbifolds}, arXiv:1604.07123.


\bibitem{FP} W. Fulton and R. Pandharipande, 
{\em Notes on stable maps and quantum cohomology}, Algebraic geometry -- Santa Cruz 1995, 
45--96, Proc. Sympos. Pure Math. {\bf 62}, Part 2: Amer. Math. Soc., Providence, RI, 1997. 

\bibitem{Gequiv} A. Givental, {\em Equivariant Gromov-Witten invariants}, Internat. Math. Res. Notices {\bf 13} (1996), 613--663.




\bibitem{Elliptic} A. Givental, {\em Elliptic Gromov-Witten invariants and the generalized mirror conjecture,}
 Integrable systems and algebraic geometry (Kobe/Kyoto, 1997), 107--155, World Sci. Publ., River Edge, NJ, 1998. 

\bibitem{SS} A. Givental, {\em Semisimple Frobenius structures
at higher genus}, Internat. Math. Res. Notices {\bf 23} (2001),  613--663.

\bibitem{GP}
T.~Graber, R.~Pandharipande,
\newblock {\em Localization of virtual classes}, Invent. Math. {\bf 135}
(1999),  487--518.


\bibitem{KL} B. Kim and  H. Lho, {\em Mirror theorem for elliptic quasimap invariants}, Geom. and Top. (to appear).


\bibitem{KZ} A. Klemm and E. Zaslow, {\em Local mirror symmetry at higher genus}, arXiv:hep-th/9906046.




\bibitem{Kon} M. Kontsevich, {\em Enumeration of rational curves via torus actions} in {\em The moduli space of curves (Texel Island, 1994)}, 335--368, Progr. Math. {\bf 129}: Birkhäuser Boston, Boston, MA, 1995. 

\bibitem{Book} Y.-P. Lee and R. Pandharipande, {\em Frobenius manifolds, Gromov-Witten theory and Virasoro constraints,} https://people.math.ethz.ch/\~{}rahul/, 2004.

\bibitem{LP2} H. Lho and R. Pandharipande,
{\em Holomorphic anomaly equations for
the formal quintic},
arXiv:1803.01409.

\bibitem{LP3} H. Lho and R. Pandharipande,
{\em Holomorphic anomaly equation for
twisted theories on projective space},
in preparation.

\bibitem{MOP} A. Marian, D. Oprea, Dragos, R. Pandharipande, 
{\em The moduli space of stable quotients,} 
Geom. Topol. {\bf 15} (2011), 1651--1706.


\bibitem{gwdt}
D.~Maulik, A. ~Oblomkov, A.~Okounkov, and R.~Pandharipande,
\newblock {\em The Gromov-{W}itten/{D}onaldson-{T}homas correspondence
for toric 3-folds}, Invent. Math. {\bf 186} (2011), 435--479.



\bibitem{MPE}  D. Maulik and R. Pandharipande, {\em New calculations in Gromov-Witten theory}, 
PAMQ {\bf{4}} (2008), 469--500.

\bibitem{ObPix} G. Oberdieck and A. Pixton, {\em Gromov-Witten theory of elliptic fibrations: Jacobi forms and holomorphic anomaly equations}, 
arXiv:1709.01481.


\bibitem{ZaZi} D. Zagier and A. Zinger, {\em Some properties of hypergeometric series associated with mirror symmetry} in {\em Modular Forms and String Duality}, 163-177, Fields Inst. Commun. {\bf 54}, AMS 2008.







\end{thebibliography}
\end{document}